\documentclass[11pt]{article}

\usepackage[T1]{fontenc}
\usepackage{lmodern}
\usepackage{microtype}
\usepackage{amsmath,amssymb,amsthm,amsfonts}
\usepackage{mathtools}
\usepackage{mathrsfs}
\usepackage{bm}
\usepackage{graphicx}
\usepackage{float}
\usepackage{multirow}
\usepackage{booktabs}
\usepackage{longtable}
\usepackage{arydshln}
\usepackage{enumerate}
\usepackage{cite}
\usepackage{aliascnt}
\usepackage{hyperref}
\usepackage[a4paper,margin=2.5cm]{geometry}
\usepackage[nameinlink,capitalise]{cleveref}

\hypersetup{
  colorlinks=true,
  linkcolor=blue,
  citecolor=blue,
  urlcolor=blue
}

\allowdisplaybreaks[2]

\numberwithin{equation}{section}

\newtheorem{theorem}{Theorem}[section]

\newaliascnt{lemma}{theorem}
\newtheorem{lemma}[lemma]{Lemma}
\aliascntresetthe{lemma}

\newaliascnt{proposition}{theorem}
\newtheorem{proposition}[proposition]{Proposition}
\aliascntresetthe{proposition}

\newaliascnt{corollary}{theorem}
\newtheorem{corollary}[corollary]{Corollary}
\aliascntresetthe{corollary}

\theoremstyle{definition}

\newaliascnt{definition}{theorem}

\aliascntresetthe{definition}

\newaliascnt{example}{theorem}

\aliascntresetthe{example}

\theoremstyle{remark}

\newaliascnt{remark}{theorem}
\newtheorem{remark}[remark]{Remark}
\aliascntresetthe{remark}

\theoremstyle{plain}

\newaliascnt{claim}{theorem}

\aliascntresetthe{claim}

\crefname{theorem}{Theorem}{Theorems}
\crefname{lemma}{Lemma}{Lemmas}
\crefname{proposition}{Proposition}{Propositions}
\crefname{corollary}{Corollary}{Corollaries}
\crefname{definition}{Definition}{Definitions}
\crefname{example}{Example}{Examples}
\crefname{remark}{Remark}{Remarks}
\crefname{claim}{Claim}{Claims}

\newcommand{\D}{\mathbb D}
\newcommand{\T}{\mathbb T}
\newcommand{\C}{\mathbb C}

\newcommand{\norm}[1]{\left\lVert #1\right\rVert}
\newcommand{\abs}[1]{\left\lvert #1\right\rvert}
\newcommand{\dd}{\,\mathrm d}
\newcommand{\opnorm}[1]{\norm{#1}_{\mathcal B^\alpha\to\mathcal B^\alpha}}

\begin{document}

\title{\Large\bf
 Norms of the Hilbert Matrix Operator on Some Analytic Function Spaces}

\author{
Puyu Cui$^{1}$,
Zhaopeng Lin$^{2}$,
Yufeng Lu$^{2}$,
Shanli Ye$^{3}$
}

\date{}
\maketitle

\footnotetext[1]{%
School of Mathematics, Liaoning Normal University,
Dalian, Liaoning 116029, China.
E-mail: \texttt{cuipuyu1234@163.com}.}

\footnotetext[2]{%
School of Mathematical Sciences,
Dalian University of Technology,
Dalian, Liaoning 116024, China.
E-mail addresses: \texttt{lyfdlut@dlut.edu.cn} (Yufeng Lu);\texttt{linzhaopeng2606@163.com} (Zhaopeng Lin).}

\footnotetext[3]{%
School of Science,
Zhejiang University of Science and Technology,
Hangzhou, Zhejiang 310023, China.
Corresponding author: Shanli Ye.
E-mail: \texttt{slye@zust.edu.cn}.}

\setcounter{footnote}{3}

\vspace{-0.8cm}

\begin{center}
\begin{minipage}{14cm}\small
\noindent{\bf Abstract}\quad
For $\mathcal H:\mathcal B^\alpha\to\mathcal B^\alpha$ with
$1<\alpha<2$, we obtain an exact norm representation. In particular, the norm is given by an explicit Gamma-function formula
for $1<\alpha\leq3/2$, whereas for $3/2<\alpha<2$ it strictly exceeds
the corresponding Gamma-function expression.  For
$\mathcal H:\mathcal B^\alpha\to\mathcal B^\alpha_{\log}$, we determine
the exact norm: it is equal to $3/2$ for $1<\alpha\leq4/3$ and is given
by an explicit Gamma-function formula for $4/3<\alpha<2$.  We also
determine the exact norm of $\mathcal H:H^1\to H^1_{\log}$ and show
that it is equal to $\pi$.  Finally, for $0<\alpha<1$, we determine
the exact norm of
$\mathcal H:H^\infty_\alpha\to\mathcal B^{\alpha+1}$, which is given
by an explicit Gamma-function formula.
\endgraf

\end{minipage}
\end{center}

\section{Introduction}
\label{sec:introduction}

Let $\C$ denote the complex plane and let $
 \D=\{z\in\C:\abs z<1\}$
be the open unit disk.  The Hilbert matrix $
 \left(\frac{1}{n+k+1}\right)_{n,k\geq0}$ 
is a classical object in operator theory.  If
$f(z)=\sum_{k\geq0}a_kz^k$ is analytic in the unit disk, its formal
action is given by
\[
 \mathcal Hf(z)
 =
 \sum_{n\geq0}
 \left(
   \sum_{k\geq0}\frac{a_k}{n+k+1}
 \right)z^n.
\]
On the analytic function spaces considered below, it is more convenient
to work with the integral form
\begin{equation}\label{eq:H-definition}
 \mathcal Hf(z)
 =
 \int_0^1\frac{f(t)}{1-tz}\,\dd t,
 \qquad z\in\D.
\end{equation}

The boundedness and norm of the Hilbert matrix operator on spaces of
analytic functions have been studied extensively.  Diamantopoulos and
Siskakis \cite{DiamantopoulosSiskakis2000} proved that $\mathcal H$ is
bounded on the Hardy spaces $H^p$, $1<p<\infty$, and obtained an upper
estimate for its norm.  Diamantopoulos \cite{Diamantopoulos2004}
subsequently considered the Bergman spaces $A^p$, $2<p<\infty$, and
derived the corresponding norm estimates.  Dostani\'c, Jevti\'c and
Vukoti\'c \cite{DostanicJevticVukotic2008} determined the exact norm of
$\mathcal H$ on $H^p$, $1<p<\infty$, and, in the Bergman setting,
obtained the exact norm on $A^p$ for $4\leq p<\infty$.  The remaining
range $2<p<4$ was later settled by Bo\v{z}in and Karapetrovi\'c
\cite{BozinKarapetrovic2018}, thereby completing the exact norm problem
on the classical Bergman spaces $A^p$ for all $2<p<\infty$.

The theory has subsequently been developed in several directions.
Galanopoulos, Girela, Pel\'aez and Siskakis
\cite{GalanopoulosGirelaPelaezSiskakis2014} studied generalized Hilbert
operators on Hardy, weighted Bergman and Dirichlet-type spaces, while
\L{}anucha, Nowak and Pavlovi\'c \cite{LanuchaNowakPavlovic2012}
investigated the Hilbert matrix on a broader class of analytic function
spaces.

An important line of research concerns the exact norm of $\mathcal H$
on the standard weighted Bergman spaces $A_\alpha^p$.  Karapetrovi\'c
\cite{Karapetrovic2018Weighted} obtained the lower bound
\[
 \norm{\mathcal H}_{A_\alpha^p\to A_\alpha^p}
 \geq
 \frac{\pi}
{\sin\left(\frac{\pi(\alpha+2)}{p}\right)} 
 =
 B\left(
 \frac{\alpha+2}{p},
 1-\frac{\alpha+2}{p}
 \right),
 \qquad -1<\alpha<p-2,
\]
and conjectured that equality holds throughout this range.  The
conjectured formula was subsequently verified for several ranges of
the parameters by Karapetrovi\'c
\cite{Karapetrovic2018Weighted,Karapetrovic2021},
Lindstr\"om, Miihkinen and Wikman
\cite{LindstromMiihkinenWikman2021}, and Dai \cite{Dai2024}.
More recently, Bao, Tian and Wulan \cite{BaoTianWulan2026} further
enlarged the range of parameters for which the formula is valid.
Subsequently, Wulan, Zhou and Zhu \cite{WulanZhouZhu2026} obtained
further exact norm results for certain even exponents and constructed
counterexamples showing that the conjectured beta-function formula
does not hold for all admissible parameters.  Thus the exact norm of
the Hilbert matrix on weighted Bergman spaces depends essentially on
the parameter range.

Related questions have also been considered on Korenblum and
logarithmically weighted spaces.  Lindstr\"om, Miihkinen and Wikman
\cite{LindstromMiihkinenWikman2019} determined the exact norm on the
Korenblum spaces $H^\infty_\alpha$ for $0<\alpha\leq2/3$ and obtained
an upper bound for $2/3<\alpha<1$.  Dai \cite{Dai2022} subsequently
obtained a norm representation for $0<\alpha<1$.  Logarithmically weighted analytic function spaces have also been
considered in connection with the Hilbert matrix operator; see
Karapetrovi\'c \cite{Karapetrovic2018Log}.
 
We now describe the results most closely related to the present paper.
For $\alpha>0$, let $\mathcal B^\alpha$ denote the $\alpha$-Bloch
space, and let $\mathcal B^\alpha_{\log}$ denote the logarithmically
weighted $\alpha$-Bloch space; their precise definitions are recalled
in Section~\ref{sec:preliminaries}.  We write $ \mathcal B=\mathcal B^1$ and
 $\mathcal B_{\log}=\mathcal B^1_{\log}$.
For background on the classical Bloch space and the $\alpha$-Bloch
spaces, see \cite{AndersonCluniePommerenke1974,Zhu1993}.

Hu and Ye \cite{HuYe2025} studied the norm of the Hilbert matrix
operator between several spaces of analytic functions.  For the
$\alpha$-Bloch space, they proved in
\cite[Theorems~5.1 and~5.2]{HuYe2025} that
\[
 \mathcal H:\mathcal B^\alpha\longrightarrow\mathcal B^\alpha
\]
is bounded for $1<\alpha<2$, and obtained the estimates
\[
 \frac{1}{2(\alpha-1)}
 \int_0^1(1-t^2)^{1-\alpha}\,\dd t
 -\frac{1}{2(\alpha-1)}
 +\frac{\pi}{\sin\bigl(\pi(\alpha-1)\bigr)}
 \leq
 \norm{\mathcal H}_{\mathcal B^\alpha\to\mathcal B^\alpha}
\]
and
\[
 \norm{\mathcal H}_{\mathcal B^\alpha\to\mathcal B^\alpha}
 \leq
 \frac{2^\alpha\pi}{\sin\bigl(\pi(\alpha-1)\bigr)}
 +\frac{1}{2-\alpha}.
\]
They also proved that $\mathcal H$ is not bounded on
$\mathcal B^\alpha$ when $0<\alpha\leq1$ or $\alpha\geq2$.

Hu and Ye also considered the Hilbert matrix operator from the
Korenblum space $H^\infty_\alpha$ into the
$(\alpha+1)$-Bloch space $\mathcal B^{\alpha+1}$.  They proved in
\cite[Theorem~7.1]{HuYe2025} that
\[
 \norm{\mathcal H}_{H^\infty_\alpha\to\mathcal B^{\alpha+1}}
 =
 \int_0^1\frac{\dd t}{(1-t^2)^\alpha}
 +
 \frac{2\pi\alpha}{\sin(\pi\alpha)},
 \qquad
 0<\alpha\leq\frac23.
\]
For $2/3<\alpha<1$, they obtained
\[
 \int_0^1\frac{\dd t}{(1-t^2)^\alpha}
 +
 \frac{2\pi\alpha}{\sin(\pi\alpha)}
 \leq
 \norm{\mathcal H}_{H^\infty_\alpha\to\mathcal B^{\alpha+1}}
\]
and
\[
 \norm{\mathcal H}_{H^\infty_\alpha\to\mathcal B^{\alpha+1}}
 \leq
 \int_0^1\frac{\dd t}{(1-t^2)^\alpha}
 +
 2\int_0^1
 \frac{(1-t)^{1-\alpha}}{t^\alpha}\,\dd t.
\]
For $\alpha\geq1$, they proved that
\[
 \mathcal H:
 H^\infty_\alpha\longrightarrow\mathcal B^{\alpha+1}
\]
is not bounded.

More recently, Ye and Zheng \cite{YeZheng2025} introduced the
logarithmically weighted Bloch and Hardy spaces relevant to the present
work and studied the Hilbert matrix operator on these spaces. With the
same normalization as used below, they proved in
\cite[Theorem~3.1]{YeZheng2025} that
\[
 \norm{\mathcal H}_{\mathcal B\to\mathcal B_{\log}}
 =\frac32.
\]

For $1<\alpha<2$, they obtained the two-sided estimate
\begin{equation}\label{eq:YZ-log-Bloch-bounds}
 \frac{1}{2(\alpha-1)}
 \int_0^1(1-t^2)^{1-\alpha}\,\dd t
 +
 \frac{3\alpha-5}{4(\alpha-1)(2-\alpha)}
 \leq
 \norm{\mathcal H}_{\mathcal B^\alpha
 \to\mathcal B^\alpha_{\log}}
 \leq
 \frac{\pi}{\sin\bigl(\pi(\alpha-1)\bigr)}
 +
 \frac{1}{2-\alpha},
\end{equation}
see \cite[Theorems~6.1 and~6.2]{YeZheng2025}. Together with the
endpoint result above, they obtained the boundedness range
\[
 \mathcal H:\mathcal B^\alpha\longrightarrow
 \mathcal B^\alpha_{\log}
 \quad\text{if and only if}\quad 1\leq\alpha<2;
\]
see \cite[Corollary~6.1]{YeZheng2025}.

For the logarithmically weighted Hardy space, Ye and Zheng proved
\cite[Corollary~4.1]{YeZheng2025} that
\[
 \pi
 \leq
 \norm{\mathcal H}_{H^1\to H^1_{\log}}
 \leq
 2\pi\log 2,
\]
and conjectured that the lower bound is sharp, namely
\cite[Conjecture~1]{YeZheng2025},
\[
 \norm{\mathcal H}_{H^1\to H^1_{\log}}=\pi.
\]

These results leave several natural sharp-norm questions open.  In
particular, the exact norms of
\[
 \mathcal H:\mathcal B^\alpha\to\mathcal B^\alpha,
 \qquad
 \mathcal H:\mathcal B^\alpha\to\mathcal B^\alpha_{\log},
 \qquad 1<\alpha<2,
\]
and $ \mathcal H:H^1\to H^1_{\log}$ 
were not known, while for
\[
 \mathcal H:H^\infty_\alpha\to\mathcal B^{\alpha+1}
\]
the exact norm remained undetermined in the range
$2/3<\alpha<1$.

We now state our main results, which settle these remaining exact norm
problems.
All function spaces and the normalization of their norms appearing
below are defined in Section~\ref{sec:preliminaries}.  Let $
 \T=\{\lambda\in\C:\abs{\lambda}=1\}$. 
For $1<\alpha<2$, set
\begin{align}
 P_\alpha(t)
 &=
 \int_0^t(1-s^2)^{-\alpha}\,\dd s,
 \qquad 0\leq t<1,
 \label{eq:P}\\
 C_\alpha
 &=
 \int_0^1(1-t)(1-t^2)^{-\alpha}\,\dd t,
 \label{eq:C}\\
 D_\alpha(r)
 &=
 (1-r^2)^\alpha
 \int_0^1\frac{tP_\alpha(t)}{(1-rt)^2}\,\dd t,
 \qquad 0\leq r<1.
 \label{eq:D}
\end{align}
We also define
\begin{equation}\label{eq:F}
 F_\alpha(z)
 =
 \int_0^z(1-\zeta^2)^{-\alpha}\,\dd\zeta.
\end{equation}
Here $(1-z^2)^{-\alpha}$ denotes the analytic branch on $\D$
which is positive on $[0,1)$, and $\Gamma$ denotes the Gamma function.

We first determine the exact norm of the Hilbert matrix operator on
$\mathcal B^\alpha$.  

\begin{theorem}
\label{thm:main}
Let $1<\alpha<2$.  Then
\begin{equation}\label{eq:exact-norm}
 \norm{\mathcal H}_{\mathcal B^\alpha\to\mathcal B^\alpha}
 =
 C_\alpha+\sup_{0\leq r<1}D_\alpha(r).
\end{equation}
The norm-attaining functions of norm one are exactly
\begin{equation}\label{eq:extremals}
 f(z)=\lambda F_\alpha(z),
 \qquad \lambda\in\T.
\end{equation}
Moreover,
\begin{equation}\label{eq:C-gamma}
 C_\alpha
 =
 \frac{1}{2(\alpha-1)}
 +
 \frac{\sqrt\pi\,\Gamma(1-\alpha)}
      {2\Gamma(3/2-\alpha)},
\end{equation}
where the gamma quotient is interpreted by continuous extension at
$\alpha=3/2$.

If $1<\alpha\leq3/2$, then
\begin{equation}\label{eq:closed-norm}
 \norm{\mathcal H}_{\mathcal B^\alpha\to\mathcal B^\alpha}
 =
 C_\alpha+
 \frac{\pi}{\sin\left(\pi(\alpha-1)\right)}.
\end{equation}
In particular,
\[
 \norm{\mathcal H}_{\mathcal B^{3/2}\to\mathcal B^{3/2}}
 =1+\pi.
\]

If $3/2<\alpha<2$, then $D_\alpha$ extends continuously to $[0,1]$
with
\begin{equation}
 D_\alpha(1)
 =
 \frac{\pi}{\sin\bigl(\pi(\alpha-1)\bigr)},
 \nonumber
\end{equation}
and
\begin{equation}\label{eq:interior-strict}
 \sup_{0\leq r<1}D_\alpha(r)>D_\alpha(1).
\end{equation}
Consequently, the supremum in \eqref{eq:exact-norm} is attained at
some $r_\alpha\in(0,1)$.
\end{theorem}

We next sharpen the two-sided estimates of Ye and Zheng
\cite[Theorems~6.1 and~6.2]{YeZheng2025} for the Hilbert matrix operator
from $\mathcal B^\alpha$ into $\mathcal B^\alpha_{\log}$ to an exact
norm formula and identify all norm-attaining functions.

\begin{theorem}
\label{thm:log-main}
Let $1<\alpha<2$.  Then
\begin{equation}\label{eq:log-exact}
 \norm{\mathcal H}_{\mathcal B^\alpha\to\mathcal B^\alpha_{\log}}
 =
 \begin{cases}
 \displaystyle\frac32,
   &1<\alpha\leq\frac43,\\[6pt]
 \displaystyle
 C_\alpha+
 \frac{\sqrt\pi\,\Gamma(2-\alpha)}
      {4\Gamma(5/2-\alpha)},
   &\frac43<\alpha<2.
 \end{cases}
\end{equation}
At $\alpha=4/3$ the two expressions agree.

The norm-attaining functions of norm one are precisely
\begin{equation}\label{eq:log-extremals}
 \begin{cases}
 f(z)=\lambda,
   &1<\alpha<4/3,\\[3pt]
 f(z)=\lambda\bigl(a+(1-a)F_{4/3}(z)\bigr),
   \quad 0\leq a\leq1,
   &\alpha=4/3,\\[3pt]
 f(z)=\lambda F_\alpha(z),
   &4/3<\alpha<2,
 \end{cases}
 \qquad \lambda\in\T.
\end{equation}
\end{theorem}

The preceding two theorems deal with Bloch-type spaces. We next turn
to the logarithmically weighted Hardy space. Ye and Zheng proved
\cite[Corollary~4.1]{YeZheng2025} that
\[
 \pi
 \leq
 \norm{\mathcal H}_{H^1\to H^1_{\log}}
 \leq
 2\pi\log 2,
\]
and formulated the conjecture
\cite[Conjecture~1]{YeZheng2025} that the exact norm is $\pi$.
The following theorem confirms this conjecture.

\begin{theorem}\label{thm:exact-H1-H1log}
The Hilbert matrix operator
\[
 \mathcal H:H^1\longrightarrow H^1_{\log}
\]
is bounded and
\[
 \norm{\mathcal H}_{H^1\to H^1_{\log}}=\pi.
\]
\end{theorem}

Finally, we return to the mapping from the Korenblum space
$H^\infty_\alpha$ into the $(\alpha+1)$-Bloch space considered by Hu
and Ye \cite[Theorem~7.1]{HuYe2025}.  They obtained the exact norm for
$0<\alpha\leq2/3$ and lower and upper bounds for $2/3<\alpha<1$.
Our last main result shows that the same norm formula holds throughout
the whole range $0<\alpha<1$.

\begin{theorem}\label{thm:KB-complete}
Let $0<\alpha<1$.  Then
\begin{equation}\label{eq:KB-complete-norm}
 \norm{\mathcal H}_{H^\infty_\alpha\to\mathcal B^{\alpha+1}}
 =
 \frac{\sqrt\pi\,\Gamma(1-\alpha)}
      {2\Gamma(3/2-\alpha)}
 +
 \frac{2\pi\alpha}{\sin(\pi\alpha)}.
\end{equation}
Moreover,
\[
 k_\alpha(z)=(1-z^2)^{-\alpha}
\]
has unit norm in $H^\infty_\alpha$ and attains the operator norm.
\end{theorem}
 
The paper is organized as follows.  In Section~\ref{sec:preliminaries},
we recall the function spaces considered in the paper and collect several
auxiliary estimates and formulas that will be used.  In
Section~\ref{sec:bloch-norm}, we determine the exact norm of the Hilbert
matrix operator on $\mathcal B^\alpha$ for $1<\alpha<2$, and characterize all
norm-attaining functions.  In Section~\ref{sec:log-target}, we determine
the exact norm of $ \mathcal H:\mathcal B^\alpha\longrightarrow
 \mathcal B^\alpha_{\log}$ for $ 1<\alpha<2$.  Section~\ref{sec:H1-H1log}
is devoted to the exact norm of $
 \mathcal H:H^1\longrightarrow H^1_{\log}$. 
Finally, in Section~\ref{sec:korenblum-bloch}, we determine the exact
norm of $
 \mathcal H:H^\infty_\alpha\longrightarrow
 \mathcal B^{\alpha+1}$ 
for the full range $0<\alpha<1$ and identify an extremal function.

\section{Preliminaries}
\label{sec:preliminaries}

\subsection{Bloch-type, Hardy and Korenblum spaces}
 
 Let $H(\D)$ denote the space of analytic functions on $\D$. 
For $\alpha>0$, the $\alpha$-Bloch space $\mathcal B^\alpha$ consists
of all $f\in H(\D)$ such that
\begin{equation}\label{eq:Bloch-norm}
 \norm{f}_{\mathcal B^\alpha}
 =
 \abs{f(0)}+\norm{f}_{\alpha,*}<\infty,
 \qquad
 \norm{f}_{\alpha,*}
 =
 \sup_{z\in\D}
 (1-\abs z^2)^\alpha\abs{f'(z)}.
\end{equation}
When $\alpha=1$, $\mathcal B^\alpha$ coincides with the classical Bloch space $\mathcal B$.

Following Ye and Zheng \cite{YeZheng2025}, set
\begin{equation}\label{eq:L-weight}
 L(r)
 =
 \log\frac{e}{(1-r)^2},
 \qquad 0\leq r<1.
\end{equation}
The logarithmically weighted $\alpha$-Bloch space
$\mathcal B^\alpha_{\log}$ is equipped with the norm
\begin{equation}\label{eq:log-Bloch-norm}
 \norm{f}_{\mathcal B^\alpha_{\log}}
 =
 \abs{f(0)}
 +
 \sup_{z\in\D}
 \frac{(1-\abs z^2)^\alpha}{L(\abs z)}
 \abs{f'(z)}.
\end{equation}
 
For $f\in H(\D)$ and $0\leq r<1$, set
\[
 M_1(r,f)
 =
 \frac{1}{2\pi}
 \int_0^{2\pi}
 \abs{f(re^{i\theta})}\,\dd\theta.
\]
The Hardy space $H^1$ consists of all $f\in H(\D)$ such that
\[
 \norm{f}_{H^1}
 =
 \sup_{0\leq r<1}M_1(r,f)
 <\infty.
\]
The logarithmically weighted Hardy space $H^1_{\log}$ consists of all
$f\in H(\D)$ such that
\[
 \norm{f}_{H^1_{\log}}
 =
 \sup_{0\leq r<1}
 \frac{M_1(r,f)}{L(r)}
 <\infty.
\]
 
For $\alpha>0$, the Korenblum space $H^\infty_\alpha$ consists of all
$f\in H(\D)$ such that
\[
 \norm{f}_{H^\infty_\alpha}
 =
 \sup_{z\in\D}
 (1-\abs z^2)^\alpha\abs{f(z)}
 <\infty.
\]

\begin{remark}
The definition of $F_\alpha$ in \eqref{eq:F} is well-defined.  Indeed,
for $z\in\D$,
\[
 \operatorname{Re}(1-z^2)\geq1-\abs z^2>0,
\]
so $1-z^2$ takes values in the right half-plane.  Hence
$(1-z^2)^{-\alpha}$ has an analytic branch on $\D$, uniquely
determined by requiring it to be positive on $[0,1)$.  Since $\D$ is
simply connected, this branch has an analytic primitive, and therefore
the integral defining $F_\alpha$ is independent of the path.
\end{remark}

\subsection{Beta and Gamma functions}

We use the standard notation
\[
 B(a,b)
 =
 \int_0^1
 x^{a-1}(1-x)^{b-1}\,\dd x,
 \qquad a,b>0,
\]
and
\[
 B(a,b)
 =
 \frac{\Gamma(a)\Gamma(b)}{\Gamma(a+b)}.
\]
We also use Euler's reflection formula
\begin{equation}\label{eq:Euler-reflection}
 \Gamma(z)\Gamma(1-z)
 =
 \frac{\pi}{\sin(\pi z)},
 \qquad z\notin\mathbb Z.
\end{equation}
The incomplete beta function is denoted by
\[
 B_x(a,b)
 =
 \int_0^x u^{a-1}(1-u)^{b-1}\,\dd u,
 \qquad 0<x<1,
\]
whenever the integral is well defined. For its standard properties and
analytic continuation, see \cite[Section~8.17]{OlverEtAl2010}.

\section{Norm of the Hilbert matrix
\texorpdfstring{$\norm{\mathcal H}_{\mathcal B^\alpha\to\mathcal B^\alpha}$}
{||H|| B-alpha to B-alpha}}
\label{sec:bloch-norm}

\subsection{Sharp radial estimate and norm reduction}

We begin with the following sharp radial estimate.

\begin{lemma}\label{lem:pointwise-estimate}
Let $1<\alpha<2$ and $f\in\mathcal B^\alpha$. Then
\begin{equation}\label{eq:primitive-bound}
 \abs{f(t)}
 \leq
 \abs{f(0)}+\norm{f}_{\alpha,*}P_\alpha(t),
 \qquad 0\leq t<1,
\end{equation}
and
\begin{equation}\label{eq:f-L1}
 \int_0^1\abs{f(t)}\,\dd t
 \leq
 \abs{f(0)}+\norm{f}_{\alpha,*}C_\alpha
 <\infty.
\end{equation}
Consequently, \eqref{eq:H-definition} defines an analytic function on $\D$,
and
\begin{equation}\label{eq:H-derivative}
 (\mathcal Hf)'(z)=\int_0^1\frac{tf(t)}{(1-tz)^2}\,\dd t.
\end{equation}
Moreover, $F_\alpha\in\mathcal B^\alpha$ with
\[
 \norm{F_\alpha}_{\mathcal B^\alpha}=1,
 \qquad F_\alpha(t)=P_\alpha(t)>0\quad(0<t<1).
\]
In particular, \eqref{eq:primitive-bound} is sharp when $f(0)=0$ and
$\norm{f}_{\alpha,*}\leq1$, since equality holds for $f=F_\alpha$ on
$(0,1)$.
\end{lemma}

\begin{proof} 
For $0\leq u<1$,
\[
 \abs{f'(u)}
 \leq
 \norm{f}_{\alpha,*}(1-u^2)^{-\alpha}.
\]
Integrating both sides over the real interval $[0,t]$, we obtain
\[
 \abs{f(t)-f(0)}
 \leq
 \norm{f}_{\alpha,*}
 \int_0^t(1-u^2)^{-\alpha}\,\dd u
 =
 \norm{f}_{\alpha,*}P_\alpha(t),
\]
and hence \eqref{eq:primitive-bound}.  Using Fubini's theorem, we have
\[
 \int_0^1P_\alpha(t)\,\dd t
 =\int_0^1\int_0^t(1-u^2)^{-\alpha}\,\dd u\,\dd t
 =\int_0^1(1-u)(1-u^2)^{-\alpha}\,\dd u=C_\alpha.
\]
Near $u=1$, the last integrand is comparable to
$2^{-\alpha}(1-u)^{1-\alpha}$, which is integrable exactly when $\alpha<2$.
This proves \eqref{eq:f-L1}.

Since $f\in L^1(0,1)$ by \eqref{eq:f-L1}, for $z\in\D$ we may expand
\[
 \frac{1}{1-tz}=\sum_{n=0}^{\infty}t^nz^n.
\]
For every $0<\rho<1$ and $\abs z\leq\rho$,
\[
 \sum_{n=0}^{\infty}
 \int_0^1 t^n\abs{f(t)}\,\abs z^n\,\dd t
 \leq
 \frac{1}{1-\rho}\int_0^1\abs{f(t)}\,\dd t<\infty.
\]
Hence the sum and the integral may be interchanged, and
\[
 \mathcal Hf(z)
 =
 \sum_{n=0}^{\infty}
 \left(\int_0^1t^nf(t)\,\dd t\right)z^n.
\]
Thus $\mathcal Hf$ is analytic in $\D$.  Differentiating the power
series term by term and interchanging the resulting sum and integral,
we obtain
\[
 (\mathcal Hf)'(z)
 =
 \int_0^1
 \sum_{n=1}^{\infty}nt^nz^{n-1}f(t)\,\dd t
 =
 \int_0^1\frac{tf(t)}{(1-tz)^2}\,\dd t,
\]
which is \eqref{eq:H-derivative}.

Since $F_\alpha(0)=0$ and $F_\alpha'(z)=(1-z^2)^{-\alpha}$,
\[
 (1-\abs z^2)^\alpha\abs{F_\alpha'(z)}
 =\left(\frac{1-\abs z^2}{\abs{1-z^2}}\right)^\alpha\leq1.
\]
Equality holds for every real $z=r\in[0,1)$, because
$\abs{1-r^2}=1-r^2$.  The identity $F_\alpha(t)=P_\alpha(t)$ for $0<t<1$ follows directly
from the definitions.
\end{proof}

For $1<\alpha<2$, set
\begin{equation}\label{eq:J}
 J_\alpha(r)
 =
 (1-r^2)^\alpha
 \int_0^1\frac{t}{(1-rt)^2}\,\dd t,
 \qquad 0\leq r<1,
\end{equation}
and write
\begin{equation}\label{eq:J-D-star}
 J_\alpha^*
 =
 \sup_{0\leq r<1}J_\alpha(r),
 \qquad
 D_\alpha^*
 =
 \sup_{0\leq r<1}D_\alpha(r).
\end{equation}

Hu and Ye \cite[Lemma~5.1 and Theorems~5.1--5.2]{HuYe2025}
obtained two-sided norm estimates by combining radial growth bounds
with separate estimates for the value at the origin and the derivative
seminorm.  We retain this idea and use the sharp radial majorant $P_\alpha$ from
\cref{lem:pointwise-estimate}. Since the resulting bounds are attained
by explicit functions, this leads to an exact norm reduction.

\begin{proposition}\label{prop:components}
For $1<\alpha<2$,
\begin{equation}\label{eq:component-formula}
 \opnorm{\mathcal H}
 =
 \max\{1+J_\alpha^*,\ C_\alpha+D_\alpha^*\}.
\end{equation} 
\end{proposition}

\begin{proof}
Let $f\in\mathcal B^\alpha$. By \eqref{eq:f-L1},
\begin{equation}\label{eq:H0-bound}
 \abs{\mathcal Hf(0)}
 \leq
 \abs{f(0)}+\norm{f}_{\alpha,*}C_\alpha.
\end{equation}
If $r=\abs z$, then $\abs{1-tz}\geq1-tr$, and
\eqref{eq:primitive-bound} gives
\begin{align}
 (1-\abs z^2)^\alpha\abs{(\mathcal Hf)'(z)}
 &\leq
 (1-r^2)^\alpha
 \int_0^1
 \frac{t\bigl(\abs{f(0)}+\norm{f}_{\alpha,*}P_\alpha(t)\bigr)}
 {(1-tr)^2}\,\dd t
 \notag\\
 &=
 \abs{f(0)}\,J_\alpha(r)
 +
 \norm{f}_{\alpha,*}D_\alpha(r).
\label{eq:Hstar-bound}
\end{align}
Taking the supremum over $z\in\D$, we obtain
\[
 \norm{\mathcal Hf}_{\alpha,*}
 \leq
 \abs{f(0)}\,J_\alpha^*
 +
 \norm{f}_{\alpha,*}D_\alpha^*.
\]
Together with \eqref{eq:H0-bound}, this yields
\[
 \norm{\mathcal Hf}_{\mathcal B^\alpha}
 \leq
 \abs{f(0)}(1+J_\alpha^*)
 +
 \norm{f}_{\alpha,*}(C_\alpha+D_\alpha^*).
\]
Since $
 \norm{f}_{\mathcal B^\alpha}
 =
 \abs{f(0)}+\norm{f}_{\alpha,*}$, 
it follows that
\[
 \norm{\mathcal Hf}_{\mathcal B^\alpha}
 \leq
 \max\{1+J_\alpha^*,\,C_\alpha+D_\alpha^*\}
 \norm{f}_{\mathcal B^\alpha}.
\]
This proves the upper bound in \eqref{eq:component-formula}.

For the reverse inequality, first take $f\equiv1$. Then
$\norm{f}_{\mathcal B^\alpha}=1$,
\[
 \mathcal H1(0)=1,
\]
and, for $0\leq r<1$,
\[
 (1-r^2)^\alpha(\mathcal H1)'(r)
 =
 J_\alpha(r).
\]
Hence $
 \norm{\mathcal H1}_{\mathcal B^\alpha}
 =
 1+J_\alpha^*$, 
and therefore
\[
 \opnorm{\mathcal H}\geq1+J_\alpha^*.
\]

Next take $f=F_\alpha$. By
\cref{lem:pointwise-estimate},
\[
 \norm{F_\alpha}_{\mathcal B^\alpha}=1,
 \qquad
 F_\alpha(t)=P_\alpha(t),\quad 0<t<1.
\]
Thus
\[
 \mathcal HF_\alpha(0)
 =
 \int_0^1P_\alpha(t)\,\dd t
 =
 C_\alpha,
\]
and, for $0\leq r<1$,
\[
 (1-r^2)^\alpha(\mathcal HF_\alpha)'(r)
 =
 D_\alpha(r).
\]
Consequently,
\[
 \norm{\mathcal HF_\alpha}_{\mathcal B^\alpha}
 =
 C_\alpha+D_\alpha^*,
\]
so that
\[
 \opnorm{\mathcal H}\geq C_\alpha+D_\alpha^*.
\]
Combining the two lower bounds with the upper bound proves
\eqref{eq:component-formula}.
\end{proof}

\subsection{Decomposition of
\texorpdfstring{$P_\alpha$}{P-alpha}
and boundary behavior of
\texorpdfstring{$D_\alpha$}{D-alpha}}

In view of Proposition~\ref{prop:components}, we now study
\[
 \sup_{0\leq r<1}D_\alpha(r).
\]
We first determine the boundary behavior of $D_\alpha(r)$ as $r\to1^-$.
For this purpose, we use the following decomposition of $P_\alpha$.

Define
\begin{equation}\label{eq:QdE}
 Q_\alpha(t)=\frac{(1-t^2)^{1-\alpha}}{2(\alpha-1)},\qquad
 d_\alpha=C_\alpha-\frac1{2(\alpha-1)},
\end{equation}
and
\begin{equation}\label{eq:E}
 E_\alpha(t)=\int_t^1(1-s)(1-s^2)^{-\alpha}\,\dd s.
\end{equation}

The definition of $Q_\alpha$ is also natural from the boundary behavior of
$P_\alpha$. Indeed,
\[
 P_\alpha'(t)=(1-t^2)^{-\alpha},
 \qquad
 Q_\alpha'(t)=t(1-t^2)^{-\alpha},
\]
and hence
\[
 (P_\alpha-Q_\alpha)'(t)
 =(1-t)(1-t^2)^{-\alpha}.
\]
The latter is integrable near $t=1$ for $1<\alpha<2$. Thus
$Q_\alpha$ captures the leading singular behavior of $P_\alpha$ at the
boundary, while $P_\alpha-Q_\alpha$ has a finite limit as $t\to1^-$.
The decomposition below makes this relation precise by separating
$P_\alpha$ into its leading term, a boundary constant, and a remainder
which vanishes as $t\to1^-$.

\begin{lemma}\label{lem:P-decomposition}
For $1<\alpha<2$ and $0\leq t<1$,
\begin{equation}\label{eq:P-decomposition}
 P_\alpha(t)=Q_\alpha(t)+d_\alpha-E_\alpha(t).
\end{equation}
Moreover,
\begin{equation}\label{eq:E-bound}
 0\leq E_\alpha(t)\leq\frac{(1-t)^{2-\alpha}}{2-\alpha},
\end{equation}
and
\begin{equation}\label{eq:d-gamma}
 d_\alpha=\frac{\sqrt\pi\,\Gamma(1-\alpha)}
 {2\Gamma(3/2-\alpha)}.
\end{equation}
In particular,
\[
 d_\alpha<0\quad(1<\alpha<3/2),\qquad
 d_{3/2}=0,
 \qquad d_\alpha>0\quad(3/2<\alpha<2).
\]
\end{lemma}

\begin{proof}
Direct differentiation gives
\[
 P_\alpha'(t)=(1-t^2)^{-\alpha},\qquad
 Q_\alpha'(t)=t(1-t^2)^{-\alpha},
\]
and therefore
\begin{equation}\label{eq:PQ-derivative}
 (P_\alpha-Q_\alpha)'(t)=(1-t)(1-t^2)^{-\alpha}.
\end{equation}
The right-hand side is integrable near $1$.  Integrating both sides of \eqref{eq:PQ-derivative} over $[0,t]$, we obtain
\[
 P_\alpha(t)-Q_\alpha(t)
 =
 -\frac{1}{2(\alpha-1)}
 +
 \int_0^t(1-s)(1-s^2)^{-\alpha}\,\dd s.
\]
Since
\[
 \int_0^t(1-s)(1-s^2)^{-\alpha}\,\dd s
 =
 C_\alpha-E_\alpha(t),
\]
this gives \eqref{eq:P-decomposition}, with $d_\alpha$ as in
\eqref{eq:QdE}.  Since
$(1+s)^{-\alpha}\leq1$,
\[
 E_\alpha(t)=\int_t^1(1-s)^{1-\alpha}(1+s)^{-\alpha}\,\dd s
 \leq\int_t^1(1-s)^{1-\alpha}\,\dd s,
\]
which is \eqref{eq:E-bound}.

Putting $x=t^2$. Since $E_\alpha(t)\to0$ as $t\to1^-$, \eqref{eq:P-decomposition} gives
\[
 d_\alpha
 =
 \frac12\lim_{x\to1^-}
 \left[
 \int_0^x u^{-1/2}(1-u)^{-\alpha}\,\dd u
 +
 \frac{(1-x)^{1-\alpha}}{1-\alpha}
 \right].
\]
A direct differentiation gives
\[
 \frac{\dd}{\dd u}
 \left[u^{1/2}(1-u)^{1-\alpha}\right]
 =
 \left(\frac32-\alpha\right)
 u^{-1/2}(1-u)^{1-\alpha}
 -
 (1-\alpha)u^{-1/2}(1-u)^{-\alpha}.
\]
Integrating over $[0,x]$, we obtain
\[
 \begin{aligned}
 &(1-\alpha)
 \int_0^x u^{-1/2}(1-u)^{-\alpha}\,\dd u =
 \left(\frac32-\alpha\right)
 \int_0^x u^{-1/2}(1-u)^{1-\alpha}\,\dd u
 -
 x^{1/2}(1-x)^{1-\alpha}.
 \end{aligned}
\]
Therefore,
\[
 \begin{aligned}
 &\int_0^x u^{-1/2}(1-u)^{-\alpha}\,\dd u
 +\frac{(1-x)^{1-\alpha}}{1-\alpha}\\
 &\qquad=
 \frac{3/2-\alpha}{1-\alpha}
 \int_0^x u^{-1/2}(1-u)^{1-\alpha}\,\dd u
 +
 \frac{(1-x^{1/2})(1-x)^{1-\alpha}}{1-\alpha}.
 \end{aligned}
\]
Since
\[
 (1-x^{1/2})(1-x)^{1-\alpha}
 =
 \frac{(1-x)^{2-\alpha}}{1+x^{1/2}}
 \longrightarrow0
 \qquad (x\to1^-),
\]
and $2-\alpha>0$, we conclude that
\[
 d_\alpha
 =
 \frac12\frac{3/2-\alpha}{1-\alpha}
 B\left(\frac12,2-\alpha\right)
 =
 \frac{\sqrt\pi\,\Gamma(1-\alpha)}
 {2\Gamma(3/2-\alpha)}.
\]
The beta representation extends continuously to $\alpha=3/2$, where
$d_{3/2}=0$, and it also shows that
\[
 d_\alpha<0\quad(1<\alpha<3/2),\qquad
 d_{3/2}=0,\qquad
 d_\alpha>0\quad(3/2<\alpha<2).
\]
\end{proof}

\begin{remark}
Notice that
\[
 Q_\alpha(t)-\frac{1}{2(\alpha-1)}
\]
is precisely the restriction to $[0,1)$ of the test function used by
Hu and Ye in the proof of
\cite[Theorem~5.1]{HuYe2025}.
By \eqref{eq:P-decomposition} and the definition of $d_\alpha$,
\[
 P_\alpha(t)
 =
 \left(
 Q_\alpha(t)-\frac{1}{2(\alpha-1)}
 \right)
 +
 C_\alpha-E_\alpha(t).
\]
Thus the test function of Hu and Ye captures the leading boundary
behavior of the sharp radial majorant $P_\alpha$, while
$C_\alpha-E_\alpha(t)$ gives the exact difference between the two
functions.
\end{remark}

Define  
\begin{equation}\label{eq:M}
 M_\alpha(r)=(1-r^2)^\alpha\int_0^1
 \frac{t(1-t^2)^{1-\alpha}}{(1-rt)^2}\,\dd t.
\end{equation}

Substituting \eqref{eq:P-decomposition} into the definition of
$D_\alpha$ gives
\begin{equation}\label{eq:D-decomposition}
 D_\alpha(r)=\frac{M_\alpha(r)}{2(\alpha-1)}
 +d_\alpha J_\alpha(r)-R_\alpha(r),
\end{equation}
where
\begin{equation}\label{eq:R}
 R_\alpha(r)=(1-r^2)^\alpha\int_0^1
 \frac{tE_\alpha(t)}{(1-rt)^2}\,\dd t.
\end{equation} 

\begin{proposition}\label{prop:D-boundary}
Let $1<\alpha<2$. Then
\begin{equation}\label{eq:M-limit}
 \lim_{r\to1^-}M_\alpha(r)=2B(2-\alpha,\alpha).
\end{equation}
Moreover,
\[
 \lim_{r\to1^-}D_\alpha(r)
 =
 \frac{\pi}{\sin\bigl(\pi(\alpha-1)\bigr)}.
\]
Hence, setting $D_\alpha(1)$ equal to this limit,
the function $D_\alpha$ extends continuously to $[0,1]$ and is bounded
there.

If $1<\alpha\leq3/2$, then
\begin{equation}\label{eq:M-domination}
 M_\alpha(r)\leq2B(2-\alpha,\alpha),
 \qquad 0\leq r<1.
\end{equation}
\end{proposition}

\begin{proof}
With the change of variables
\[
 y=\frac{1-t}{1-rt},
 \qquad
 t=\frac{1-y}{1-ry},
\]
Thus
\begin{equation}\label{eq:M-transformed}
 M_\alpha(r)
 =
 \int_0^1 m_{\alpha,r}(y) \,\dd y,
\end{equation}
where
\[
m_{\alpha,r}(y)
 =
 (1+r)^\alpha
 (1-y)y^{1-\alpha}
 [2-(1+r)y]^{1-\alpha}
 (1-ry)^{2\alpha-3},
 \qquad 0<y<1.
\] If $1<\alpha\leq3/2$, then for $0<y<1$,
\[
 \frac{m_{\alpha,r}(y)}{m_{\alpha,1}(y)}
 =\left(\frac{1+r}{2}\right)^\alpha
 \left(\frac{2-(1+r)y}{2(1-y)}\right)^{1-\alpha}
 \left(\frac{1-ry}{1-y}\right)^{2\alpha-3}\leq1.
\]
Integration gives \eqref{eq:M-domination}. Moreover,
$m_{\alpha,1}$ is integrable on $(0,1)$ and dominates
$m_{\alpha,r}$ for all $0\le r<1$.  If
$3/2\leq\alpha<2$, use $1-ry\leq1$ and
$2-(1+r)y\geq2(1-y)$,
\[
 m_{\alpha,r}(y)
 \leq
 2y^{1-\alpha}(1-y)^{2-\alpha},
 \qquad 0<y<1.
\]
Dominated convergence now gives
\[
 \lim_{r\to1^-}M_\alpha(r)
 =2\int_0^1y^{1-\alpha}(1-y)^{\alpha-1}\,\dd y,
\]
which is \eqref{eq:M-limit}.

We next determine the boundary limit of $D_\alpha$.  Since $
 0\leq E_\alpha(t)\leq E_\alpha(0)=C_\alpha$, 
the definition of $R_\alpha$ gives
\begin{equation}\label{eq:R-J}
 0\leq R_\alpha(r)\leq C_\alpha J_\alpha(r).
\end{equation}
For $r>0$, direct integration yields
\begin{equation}\label{eq:J-integral}
 \int_0^1\frac{t}{(1-rt)^2}\,\dd t
 =
 \frac{r/(1-r)+\log(1-r)}{r^2}.
\end{equation}
Consequently,
\[
 J_\alpha(r)
 =
 (1-r^2)^\alpha
 \frac{r/(1-r)+\log(1-r)}{r^2}
 \longrightarrow0
 \qquad (r\to1^-),
\]
because $\alpha>1$.  It follows from \eqref{eq:R-J} that
\[
 R_\alpha(r)\longrightarrow0
 \qquad (r\to1^-).
\]

Using \eqref{eq:D-decomposition} and \eqref{eq:M-limit}, we therefore obtain
\begin{align*}
 \lim_{r\to1^-}D_\alpha(r)
 &=
 \frac{B(2-\alpha,\alpha)}{\alpha-1}=
 \Gamma(2-\alpha)\Gamma(\alpha-1)= 
 \frac{\pi}{\sin\bigl(\pi(\alpha-1)\bigr)},
\end{align*}
where the last equality follows from
\eqref{eq:Euler-reflection}.

Finally, since $D_\alpha$ is continuous on $[0,1)$ and has a finite limit
as $r\to1^-$, setting $D_\alpha(1)$ equal to this limit gives a continuous
extension to $[0,1]$. Hence $D_\alpha$ is bounded on $[0,1]$.
\end{proof}
 
\begin{corollary}\label{cor:boundary-regime}
If $1<\alpha\leq3/2$, then
\[
 D_\alpha^*
 =
 \frac{\pi}{\sin\bigl(\pi(\alpha-1)\bigr)}.
\]
\end{corollary}

\begin{proof}
By Lemma~\ref{lem:P-decomposition}, $d_\alpha\leq0$ and
$E_\alpha\geq0$, so $
 P_\alpha(t)\leq Q_\alpha(t)$. 
Hence, by \eqref{eq:M-domination},
\[
 D_\alpha(r)
 \leq
 \frac{M_\alpha(r)}{2(\alpha-1)}
 \leq
 \frac{B(2-\alpha,\alpha)}{\alpha-1}
 =
 \frac{\pi}{\sin\bigl(\pi(\alpha-1)\bigr)}.
\]
The reverse inequality follows from
\[
 \lim_{r\to1^-}D_\alpha(r)
 =
 \frac{\pi}{\sin\bigl(\pi(\alpha-1)\bigr)}.
\]
\end{proof}
 
\begin{proposition}\label{prop:first-correction}
Let $3/2<\alpha<2$. As $r\to1^-$,
\begin{equation}\label{eq:D-correction}
 D_\alpha(r)
 =
 \frac{\pi}{\sin(\pi(\alpha-1))}
 +2^\alpha d_\alpha(1-r)^{\alpha-1}
 +o\bigl((1-r)^{\alpha-1}\bigr).
\end{equation}
In particular, \eqref{eq:interior-strict} holds and $D_\alpha$ attains its
maximum at some point of $(0,1)$.
\end{proposition}

\begin{proof}
Recall from \eqref{eq:M-transformed} that
\[
 M_\alpha(r)=\int_0^1m_{\alpha,r}(y)\,\dd y.
\]
 
For $r\in[1/2,1)$ and $0<y<1$, we have
\[
 \frac12\bigl((1-r)+(1-y)\bigr)
 \leq 1-ry
 \leq (1-r)+(1-y),
\]
and
\[
 (1-r)+(1-y)
 \leq 2-(1+r)y
 \leq 2\bigl((1-r)+(1-y)\bigr).
\]
Since $3/2<\alpha<2$, we have $1-\alpha<0$ and
$2\alpha-3>0$.  Hence, using the explicit formula for
$m_{\alpha,r}$,
\begin{align*}
 m_{\alpha,r}(y)
 &\leq
 C y^{1-\alpha}(1-y)
 \bigl((1-r)+(1-y)\bigr)^{1-\alpha}
 \bigl((1-r)+(1-y)\bigr)^{2\alpha-3}\\
 &=
 C y^{1-\alpha}(1-y)
 \bigl((1-r)+(1-y)\bigr)^{\alpha-2}.
\end{align*}
On the other hand, differentiating $\log m_{\alpha,r}(y)$ with respect to $r$, we obtain
\[
 \left|
 \frac{\partial_r m_{\alpha,r}(y)}
      {m_{\alpha,r}(y)}
 \right|=  \left|
 \frac{\alpha}{1+r}
 +\frac{(\alpha-1)y}{2-(1+r)y}
 -\frac{(2\alpha-3)y}{1-ry} \right|
 \leq
 C\left(
 1+\frac{1}{(1-r)+(1-y)}
 \right).
\]
Since $
 0<(1-r)+(1-y)\leq2$, we have 
\[
 \left|
 \frac{\partial_r m_{\alpha,r}(y)}
      {m_{\alpha,r}(y)}
 \right|
 \leq
 \frac{C}{(1-r)+(1-y)}.
\]
Hence
\[
 \bigl|\partial_r m_{\alpha,r}(y)\bigr|
 \leq
 C y^{1-\alpha}(1-y)
 \bigl((1-r)+(1-y)\bigr)^{\alpha-3}.
\]
Finally, since $\alpha-3<0$ and $
 (1-r)+(1-y)\geq1-y$, 
we obtain
\[
 \bigl((1-r)+(1-y)\bigr)^{\alpha-3}
 \leq
 (1-y)^{\alpha-3},
\]
and therefore
\begin{equation}\label{eq:m-derivative-majorant}
 \bigl|\partial_r m_{\alpha,r}(y)\bigr|
 \leq
 C y^{1-\alpha}(1-y)^{\alpha-2},
 \qquad 0<y<1,
\end{equation}
uniformly for $r\in[1/2,1)$.
Thus \eqref{eq:m-derivative-majorant} gives an integrable majorant for $\partial_r m_{\alpha,r}$. Hence
\[
 M_\alpha'(r)
 =
 \int_0^1\partial_r m_{\alpha,r}(y)\,\dd y
 =
 O(1),
 \qquad r\to1^-.
\]
Together with \eqref{eq:M-limit}, this yields
\[
 \frac{M_\alpha(r)}{2(\alpha-1)}
 =
 \frac{B(2-\alpha,\alpha)}{\alpha-1}
 +O(1-r)
 =
 \frac{\pi}{\sin(\pi(\alpha-1))}
 +O(1-r).
\]

Next, by \eqref{eq:J-integral},
\begin{align*}
 J_\alpha(r)
 &=
 (1-r^2)^\alpha
 \frac{r/(1-r)+\log(1-r)}{r^2} =
 2^\alpha(1-r)^{\alpha-1}(1+o(1)).
\end{align*}

For $R_\alpha$, \eqref{eq:E-bound} gives $
 E_\alpha(t)\leq C(1-t)^{2-\alpha}$. 
For $r\in[1/2,1)$,
\[
 1-rt
 =
 (1-r)+r(1-t)
 \geq
 \frac12\bigl((1-r)+(1-t)\bigr),
\]
and $
 1-r^2\leq2(1-r)$. 
Therefore
\begin{align*}
 R_\alpha(r)
 &\leq
 C(1-r)^\alpha
 \int_0^1
 \frac{(1-t)^{2-\alpha}}
 {\bigl((1-r)+(1-t)\bigr)^2}\,\dd t\\
 &=
 C(1-r)
 \int_0^{1/(1-r)}
 \frac{s^{2-\alpha}}{(1+s)^2}\,\dd s\\
 &=
 O(1-r),
\end{align*}
because
\[
 \int_0^\infty
 \frac{s^{2-\alpha}}{(1+s)^2}\,\dd s<\infty
 \qquad (1<\alpha<2).
\]

Substituting these estimates into \eqref{eq:D-decomposition}, we obtain
\[
 D_\alpha(r)
 =
 \frac{\pi}{\sin(\pi(\alpha-1))}
 +2^\alpha d_\alpha(1-r)^{\alpha-1}(1+o(1))
 +O(1-r).
\]
Since $\alpha<2$,
\[
 1-r=o\bigl((1-r)^{\alpha-1}\bigr),
\]
and hence
\[
 D_\alpha(r)
 =
 \frac{\pi}{\sin(\pi(\alpha-1))}
 +2^\alpha d_\alpha(1-r)^{\alpha-1}
 +o\bigl((1-r)^{\alpha-1}\bigr).
\]
This proves \eqref{eq:D-correction}.

Finally, $d_\alpha>0$ for $3/2<\alpha<2$, so
\eqref{eq:D-correction} implies
\[
 D_\alpha(r)>D_\alpha(1)
\]
for all $r<1$ sufficiently close to $1$. Hence
\eqref{eq:interior-strict} holds. Since $D_\alpha$ is continuous on
$[0,1]$, it attains its maximum there, and the maximum is not attained at
$r=1$. Moreover,
\[
 D_\alpha'(0)
 =
 2\int_0^1t^2P_\alpha(t)\,\dd t>0,
\]
so it is not attained at $r=0$. Therefore $D_\alpha$ attains its
maximum at some point of $(0,1)$.
\end{proof}
 
\subsection{Proof of Theorem~\ref{thm:main}}

\begin{proof}
We first compare the two quantities in \eqref{eq:component-formula}.
Since $\alpha>1$,
\[
 (1-r^2)^\alpha\leq 1-r^2,
\]
and
\[
 \int_0^1\frac{t}{(1-rt)^2}\,\dd t
 \leq
 \int_0^1\frac{1}{(1-rt)^2}\,\dd t
 =\frac{1}{1-r}.
\]
Hence
\[
 J_\alpha(r)
 \leq
 (1-r^2)\frac{1}{1-r}
 =1+r<2,
 \qquad 0\leq r<1,
\]
and therefore
\[
 J_\alpha^*\leq2.
\]
On the other hand, \cref{prop:D-boundary} gives
$D_\alpha^*<\infty$ and
\[
 D_\alpha^*
 \geq D_\alpha(1)
 =
 \frac{\pi}{\sin\bigl(\pi(\alpha-1)\bigr)}
 \geq\pi.
\]
Thus \cref{prop:components} implies that
$\mathcal H$ is bounded on $\mathcal B^\alpha$. Moreover, since
$C_\alpha>0$,
\[
 1+J_\alpha^*
 \leq3
 <\pi
 \leq D_\alpha^*
 <C_\alpha+D_\alpha^*.
\]
Consequently,
\[
 \opnorm{\mathcal H}
 =
 C_\alpha+D_\alpha^*
 =
 C_\alpha+\sup_{0\leq r<1}D_\alpha(r),
\]
which proves \eqref{eq:exact-norm}. 
Combining \eqref{eq:QdE} with \eqref{eq:d-gamma} gives
\eqref{eq:C-gamma}.

We next characterize the norm-attaining functions of norm one.
Suppose that $
 \norm{f}_{\mathcal B^\alpha}=1$ 
and
\[
 \norm{\mathcal Hf}_{\mathcal B^\alpha}
 =
 C_\alpha+D_\alpha^*.
\]
Since $
 \abs{f(0)}+\norm{f}_{\alpha,*}=1$, 
the estimate in the proof of \cref{prop:components} gives
\begin{align*}
 \norm{\mathcal Hf}_{\mathcal B^\alpha}
 &\leq
 \abs{f(0)}\bigl(1+J_\alpha^*\bigr)
 +
 \norm{f}_{\alpha,*}\bigl(C_\alpha+D_\alpha^*\bigr)\\
 &=
 C_\alpha+D_\alpha^*
 -
 \abs{f(0)}
 \bigl(C_\alpha+D_\alpha^*-1-J_\alpha^*\bigr).
\end{align*}
Since
\[
 C_\alpha+D_\alpha^*-1-J_\alpha^*>0,
\]
equality forces
\[
 f(0)=0,
 \qquad
 \norm{f}_{\alpha,*}=1.
\]

It follows from the estimates in the proof of
\cref{prop:components} that
\[
 \abs{\mathcal Hf(0)}\leq C_\alpha,
 \qquad
 \norm{\mathcal Hf}_{\alpha,*}\leq D_\alpha^*.
\]
Since their sum is $C_\alpha+D_\alpha^*$, equality must hold in both
estimates. In particular,
\[
 \abs{\mathcal Hf(0)}=C_\alpha.
\]
Since $f(0)=0$ and $\norm{f}_{\alpha,*}=1$,
\eqref{eq:primitive-bound} gives
\[
 \abs{f(t)}\leq P_\alpha(t),
 \qquad 0\leq t<1.
\]
Therefore
\[
 C_\alpha
 =
 \abs{\int_0^1f(t)\,\dd t}
 \leq
 \int_0^1\abs{f(t)}\,\dd t
 \leq
 \int_0^1P_\alpha(t)\,\dd t
 =
 C_\alpha.
\]
Therefore both inequalities above are equalities. Since $
 P_\alpha(t)-\abs{f(t)}\geq0$, 
we obtain
\[
 \abs{f(t)}=P_\alpha(t)
 \qquad\text{for almost every }t\in(0,1).
\]

Choose $\lambda\in\T$ such that
\[
 \overline{\lambda}\int_0^1f(t)\,\dd t
 =
 \abs{\int_0^1f(t)\,\dd t}.
\]
Equality in the integral triangle inequality gives
\[
 \int_0^1
 \left(
 \abs{f(t)}
 -
 \operatorname{Re}\bigl(\overline{\lambda}f(t)\bigr)
 \right)\,\dd t
 =0.
\]
The integrand is nonnegative, so it vanishes almost everywhere. Hence
\[
 f(t)
 =
 \lambda\abs{f(t)}
 =
 \lambda P_\alpha(t)
\]
for almost every $t\in(0,1)$. By continuity, the identity holds for
every $t\in(0,1)$. Since
\[
 P_\alpha(t)=F_\alpha(t),
 \qquad 0<t<1,
\]
the identity theorem yields
\[
 f=\lambda F_\alpha
 \qquad\text{on }\D.
\]
Conversely, for every $\lambda\in\T$,
$\lambda F_\alpha$ has unit norm and attains the operator norm by
\cref{lem:pointwise-estimate,prop:components}. This proves
\eqref{eq:extremals}.

It remains to identify $D_\alpha^*$ in the two parameter ranges.
If $1<\alpha\leq3/2$, then
\cref{cor:boundary-regime} and \eqref{eq:exact-norm} give
\eqref{eq:closed-norm}. At $\alpha=3/2$, we have
$d_{3/2}=0$, and hence
\[
 C_{3/2}
 =
 \frac{1}{2(3/2-1)}
 =1
\]
by \eqref{eq:QdE}. Therefore
\[
 \norm{\mathcal H}_{\mathcal B^{3/2}\to\mathcal B^{3/2}}
 =
 1+\pi.
\]

Finally, suppose that $3/2<\alpha<2$.
By \cref{prop:D-boundary}, $D_\alpha$ extends continuously to $[0,1]$
and
\[
 D_\alpha(1)
 =
 \frac{\pi}{\sin\bigl(\pi(\alpha-1)\bigr)}.
\]
Moreover, \cref{prop:first-correction} gives
\[
 \sup_{0\leq r<1}D_\alpha(r)>D_\alpha(1)
\]
and shows that the supremum is attained at some
$r_\alpha\in(0,1)$. This completes the proof.
\end{proof}

\section{Norm of the Hilbert matrix
\texorpdfstring{$\norm{\mathcal H}_{\mathcal B^\alpha\to\mathcal B^\alpha_{\log}}$}
{||H|| B-alpha to B-alpha-log}}
\label{sec:log-target}

Ye and Zheng \cite[Theorems~6.1 and~6.2]{YeZheng2025} proved the
bounds in \eqref{eq:YZ-log-Bloch-bounds}. We now determine the exact
norm.

Define
\begin{equation}\label{eq:log-components}
 K_\alpha^{(0)}
 =
 \sup_{0\leq r<1}\frac{J_\alpha(r)}{L(r)},
 \qquad
 K_\alpha^{(1)}
 =
 \sup_{0\leq r<1}\frac{D_\alpha(r)}{L(r)}.
\end{equation}
Let $f\in\mathcal B^\alpha$. Exactly as in \eqref{eq:H0-bound},
\[
 \abs{\mathcal Hf(0)}
 \leq
 \abs{f(0)}+\norm{f}_{\alpha,*}C_\alpha.
\]
Moreover, by \eqref{eq:H-derivative} and \eqref{eq:primitive-bound},
if $r=\abs z$, then
\begin{align*}
 \frac{(1-\abs z^2)^\alpha}{L(\abs z)}
 \abs{(\mathcal Hf)'(z)}
 &\leq
 \frac{(1-r^2)^\alpha}{L(r)}
 \int_0^1
 \frac{t\bigl(\abs{f(0)}
 +\norm{f}_{\alpha,*}P_\alpha(t)\bigr)}
 {(1-rt)^2}\,\dd t\\
 &=
 \abs{f(0)}\,\frac{J_\alpha(r)}{L(r)}
 +
 \norm{f}_{\alpha,*}\frac{D_\alpha(r)}{L(r)}.
\end{align*}
Since $
 \norm{f}_{\mathcal B^\alpha}
 =
 \abs{f(0)}+\norm{f}_{\alpha,*}$, 
it follows that 
\begin{equation}\label{eq:log-component-formula}
 \norm{\mathcal H}_{\mathcal B^\alpha\to\mathcal B^\alpha_{\log}}
 =
 \max\left\{
 1+K_\alpha^{(0)},
 C_\alpha+K_\alpha^{(1)}
 \right\}.
\end{equation}
Indeed, the upper bound follows from the preceding estimates, while the
equality is obtained for 
$f\equiv1$ and $f=F_\alpha$, respectively.

For every $\alpha>1$,
\begin{equation}\label{eq:K0-half}
 K_\alpha^{(0)}=\frac12.
\end{equation}
Moreover, the maximum is attained only at $r=0$. 
Indeed, since $\alpha>1$ and $0<1-r^2\leq1$,  we have 
\[
 J_\alpha(r)\leq J_1(r),
 \qquad 0\leq r<1.
\]
For $r>0$, \eqref{eq:J-integral} gives
\[
 J_1(r)
 =
 (1-r^2)
 \frac{r/(1-r)+\log(1-r)}{r^2}.
\]
Since $
 L(r)=1-2\log(1-r)$, 
a direct computation yields
\begin{equation}\label{eq:J1-log-gap}
 \frac{L(r)}2-J_1(r)
 =
 \frac{-r-r^2/2-\log(1-r)}{r^2}.
\end{equation}
Using
\[
 -\log(1-r)
 =
 \sum_{n=1}^{\infty}\frac{r^n}{n}
 >
 r+\frac{r^2}{2},
 \qquad 0<r<1,
\]
we obtain
\[
 J_1(r)<\frac{L(r)}2,
 \qquad 0<r<1.
\]
Therefore,
\[
 \frac{J_\alpha(r)}{L(r)}
 \leq
 \frac{J_1(r)}{L(r)}
 <\frac12,
 \qquad 0<r<1.
\]
On the other hand,
\[
 J_\alpha(0)=\int_0^1t\,\dd t=\frac12,
 \qquad
 L(0)=1.
\]
Thus
\[
 \frac{J_\alpha(0)}{L(0)}=\frac12,
\]
which proves \eqref{eq:K0-half} and shows that $r=0$ is the unique
maximizer.

For the interpolation argument below, we extend the preceding notation
to the endpoint $\alpha=1$ by setting
\[
 P_1(t)=\operatorname{arctanh}t,
 \qquad
 C_1=\int_0^1\frac{1-t}{1-t^2}\,\dd t=\log2,
\]
and
\[
 D_1(r)
 =
 (1-r^2)\int_0^1
 \frac{tP_1(t)}{(1-rt)^2}\,\dd t,
 \qquad 0\leq r<1.
\]

At the endpoint $\alpha=1$, the function
\[
 F_1(z)=\operatorname{arctanh}z
       =\frac12\log\frac{1+z}{1-z}
\]
is exactly the logarithmic test function used by Ye and Zheng in
\cite[Lemma~3.1]{YeZheng2025} for their study of
$\mathcal H:\mathcal B\to\mathcal B_{\log}$.
The estimate for the corresponding second component in the proof of
\cite[Lemma~3.2]{YeZheng2025}, rewritten in the present notation,
gives the following endpoint bound.

\begin{lemma}\label{lem:alpha1-log-bound}
One has
\begin{equation}\label{eq:D1-log2}
 \sup_{0\leq r<1}\frac{D_1(r)}{L(r)}
 \leq\log2.
\end{equation}
\end{lemma}

\begin{proof}
The substitution $y=(1-t)/(1-rt)$ followed by replacing $1-y$ by $t$
gives
\begin{equation}\label{eq:D1-transformed}
 D_1(r)=\frac{1+r}{2}\int_0^1
 \frac{t}{1-(1-t)r}
 \log\frac{1-r+(1+r)t}{(1-r)(1-t)}\,\dd t.
\end{equation}
For the integral in \eqref{eq:D1-transformed}, put
\[
 x=\frac{1-r+(1+r)t}{(1-r)(1-t)}.
\] 
A direct calculation yields
\begin{align*}
 &\int_0^1\frac{t}{1-(1-t)r}
 \log\frac{1-r+(1+r)t}{(1-r)(1-t)}\,\dd t\\
 &\qquad=
 \int_1^\infty\frac{x-1}{x+1}
 \frac{2(1-r)}{[1+r+(1-r)x]^2}\log x\,\dd x\\
 &\qquad\leq
 \frac{2}{1+r}\log\frac{2}{1-r}.
\end{align*}
For the final estimate, drop the factor $(x-1)/(x+1)\leq1$ and integrate by
parts.  Therefore
\[
 \frac{D_1(r)}{L(r)}
 \leq\frac{\log2-\log(1-r)}{1-2\log(1-r)}
 \leq\log2,
\]
because $2\log2-1>0$.
\end{proof}

The following coefficient estimate will be used for
$4/3\leq\alpha<2$. 
Set
\begin{equation}\label{eq:D0-formula}
 D_\alpha(0)
 =\frac12\int_0^1(1-s^2)^{1-\alpha}\,\dd s
 =\frac{\sqrt\pi\,\Gamma(2-\alpha)}
 {4\Gamma(5/2-\alpha)}.
\end{equation}

\begin{lemma} \label{lem:coefficient-domination}
If $4/3\leq\alpha<2$, then
\begin{equation}\label{eq:D-over-L-zero}
 \frac{D_\alpha(r)}{L(r)}\leq D_\alpha(0) ,
 \qquad 0\leq r<1,
\end{equation}
and the inequality is strict for $r>0$.
\end{lemma}

\begin{proof}
Since
\[
 \frac{D_\alpha(r)}{(1-r^2)^\alpha}
 =
 \int_0^1\frac{tP_\alpha(t)}{(1-rt)^2}\,\dd t
 =
 \sum_{n=0}^\infty q_nr^n,
\]
where
\[
 q_n
 =
 (n+1)\int_0^1t^{n+1}P_\alpha(t)\,\dd t.
\]
Write 
\begin{equation}\label{eq:c-generating}
 L(r)(1-r^2)^{-\alpha}
 =
 \sum_{n=0}^\infty c_nr^n.
\end{equation}
It is enough to prove
\begin{equation}\label{eq:coefficient-target}
 q_n\leq D_\alpha(0)c_n,
 \qquad n\geq0.
\end{equation}
Indeed, summing these inequalities gives \eqref{eq:D-over-L-zero}.

For $a\in\C$ and $m\geq0$, let $(a)_m$ denote the Pochhammer symbol, and
set $
 A_m=\frac{(\alpha)_m}{m!}$. 
Since
\[
 (1-r^2)^{-\alpha}
 =
 \sum_{m=0}^\infty A_mr^{2m}
\]
and
\[
 L(r)=1+2\sum_{j=1}^\infty\frac{r^j}{j},
\]
comparison of coefficients gives
\begin{align}
 c_{2m}
 &=
 A_m+\sum_{j=1}^m\frac{A_{m-j}}j
 =
 A_m\left(
 1+\sum_{k=0}^{m-1}\frac1{\alpha+k}
 \right),
 \label{eq:c-even}\\
 c_{2m+1}
 &=
 2\sum_{j=0}^m\frac{A_{m-j}}{2j+1}.
 \label{eq:c-odd}
\end{align}

To justify the second equality in \eqref{eq:c-even}, note that
\[
 -\log(1-x)(1-x)^{-\alpha}
 =
 \frac{\partial}{\partial\alpha}(1-x)^{-\alpha}.
\]
Since
\[
 (1-x)^{-\alpha}=\sum_{m=0}^\infty A_mx^m,
\]
comparison of the coefficients of $x^m$ gives
\[
 \sum_{j=1}^m\frac{A_{m-j}}j
 =
 \frac{\partial A_m}{\partial\alpha}
 =
 A_m\sum_{k=0}^{m-1}\frac1{\alpha+k}.
\]

We first prove \eqref{eq:coefficient-target} for the even indices.
The identity
\[
 2tP_\alpha(t)
 =
 (1-t^2)^{1-\alpha}
 -
 \frac{\dd}{\dd t}
 \bigl((1-t^2)P_\alpha(t)\bigr)
\]
and
\[
 (1-t^2)P_\alpha(t)\longrightarrow0
 \qquad (t\to1^-)
\]
give, after integration by parts,
\[
 q_0=D_\alpha(0),
\]
and, for $n\geq2$,
\[
 q_n
 =
 \frac{n+1}{n+2}
 \int_0^1t^n(1-t^2)^{1-\alpha}\,\dd t
 +
 \frac{n(n+1)}{(n+2)(n-1)}q_{n-2}.
\]
Moreover,
\[
 \frac{1}{D_\alpha(0)}
 \int_0^1t^{2m}(1-t^2)^{1-\alpha}\,\dd t
 =
 2\frac{(1/2)_m}{(5/2-\alpha)_m}.
\]
Hence, for $m\geq1$,
\[
 \frac{q_{2m}}{D_\alpha(0)A_m}
 =
 \frac{2m+1}{m+1}
 \frac{(1/2)_m}
 {(5/2-\alpha)_mA_m}
 +
 \beta_m
 \frac{q_{2m-2}}{D_\alpha(0)A_{m-1}},
\]
where
\[
 \beta_m
 =
 \frac{m^2(2m+1)}
 {(m+1)(2m-1)(\alpha+m-1)}.
\]
For $m=1$, direct calculation gives 
\[
 \frac{c_2}{A_1}
 -
 \frac{q_2}{D_\alpha(0)A_1}
 =
 \frac{\alpha+1}{\alpha}
 -
 \frac{3(3-\alpha)}{\alpha(5-2\alpha)}
 =
 \frac{2(2-\alpha)(\alpha-1)}
 {\alpha(5-2\alpha)}
 >0.
\] 
For $m\geq2$ and $\alpha\geq4/3$, we have
\[
 0<\beta_m\leq1.
\]
Indeed,
\[
 1-\beta_m
 =
 \frac{
 \frac13(2m^2-2m-1)
 +
 \left(\alpha-\frac43\right)(2m^2+m-1)
 }{
 (m+1)(2m-1)(\alpha+m-1)
 }
 \geq0.
\] 

We claim that, for $m\geq2$,
\begin{equation}\label{eq:first-term-bound}
 \frac{2m+1}{m+1}
 \frac{(1/2)_m}
 {(5/2-\alpha)_mA_m}
 \leq
 1-\beta_m+\frac1{\alpha+m-1}.
\end{equation}
Indeed,
\[
 \frac{(1/2)_m}
 {(5/2-\alpha)_mA_m}
 =
 \prod_{k=0}^{m-1}
 \frac{(k+1/2)(k+1)}
 {(k+\alpha)(k+5/2-\alpha)}.
\]

Suppose first that $
 \frac43\leq\alpha\leq\frac32$. 
Since
\[
 (k+\alpha)(k+5/2-\alpha)
 -(k+1)(k+3/2)
 =
 \frac{(\alpha-1)(3-2\alpha)}2
 \geq0,
\]
we have
\[
 \frac{(k+1/2)(k+1)}
 {(k+\alpha)(k+5/2-\alpha)}
 \leq
 \frac{k+1/2}{k+3/2}.
\]
Therefore 
\[
 \frac{(1/2)_m}
 {(5/2-\alpha)_mA_m}
 \leq
 \prod_{k=0}^{m-1}
 \frac{k+1/2}{k+3/2}
 =
 \frac1{2m+1}.
\]
It follows that
\[
 \frac{2m+1}{m+1}
 \frac{(1/2)_m}
 {(5/2-\alpha)_mA_m}
 \leq
 \frac1{m+1}.
\]
On the other hand,
\begin{align*}
 &1-\beta_m+\frac1{\alpha+m-1}-\frac1{m+1}\\
 &\qquad=
 \frac{
 2\alpha m^2-\alpha m-2m^2+2m-1
 }{
 (m+1)(2m-1)(\alpha+m-1)
 }\\
 &\qquad=
 \frac{
 \frac13(2m^2+2m-3)
 +
 \left(\alpha-\frac43\right)m(2m-1)
 }{
 (m+1)(2m-1)(\alpha+m-1)
 }
 >0.
\end{align*}
Thus \eqref{eq:first-term-bound} holds in this range.

Now suppose that $
 \frac32<\alpha<2$.
Since
\[
 (k+\alpha-1)(k+5/2-\alpha)
 -(k+1/2)(k+1)
 =
 \frac{(2-\alpha)(2\alpha-3)}2
 >0,
\]
we obtain
\[
 \frac{(k+1/2)(k+1)}
 {(k+\alpha)(k+5/2-\alpha)}
 \leq
 \frac{k+\alpha-1}{k+\alpha}.
\]
Thus
\[
 \frac{(1/2)_m}
 {(5/2-\alpha)_mA_m}
 \leq
 \prod_{k=0}^{m-1}
 \frac{k+\alpha-1}{k+\alpha}
 =
 \frac{\alpha-1}{\alpha+m-1}.
\]
Hence
\[
 \frac{2m+1}{m+1}
 \frac{(1/2)_m}
 {(5/2-\alpha)_mA_m}
 \leq
 \frac{(2m+1)(\alpha-1)}
 {(m+1)(\alpha+m-1)}.
\]
Finally,
\begin{align*}
 &1-\beta_m+\frac1{\alpha+m-1}
 -
 \frac{(2m+1)(\alpha-1)}
 {(m+1)(\alpha+m-1)}
 \\
 &\qquad=
 \frac{
 (2-\alpha)(2m^2-m)+(m-1)
 }{
 (m+1)(2m-1)(\alpha+m-1)
 }
 >0.
\end{align*}
This proves \eqref{eq:first-term-bound}.

We now prove the even coefficient estimate by induction.  Suppose that
\[
 q_{2m-2}\leq D_\alpha(0)c_{2m-2}.
\]
Using \eqref{eq:c-even}, \eqref{eq:first-term-bound}, and
$0<\beta_m\leq1$, we obtain
\begin{align*}
 \frac{q_{2m}}{D_\alpha(0)A_m}
 &\leq
 1-\beta_m+\frac1{\alpha+m-1}
 +\beta_m
 \left(
 1+\sum_{k=0}^{m-2}\frac1{\alpha+k}
 \right)\\
 &=
 1+\frac1{\alpha+m-1}
 +\beta_m
 \sum_{k=0}^{m-2}\frac1{\alpha+k}\\
 &\leq
 1+\sum_{k=0}^{m-1}\frac1{\alpha+k}
 =
 \frac{c_{2m}}{A_m}.
\end{align*}
Together with the cases $m=0,1$, this proves
\eqref{eq:coefficient-target} for every even index.

For the odd indices, directly from the definition of $q_n$,
\[
 q_{2m+1}
 <
 \frac{2m+2}{2m+1}q_{2m}.
\]
On the other hand, \eqref{eq:c-even}--\eqref{eq:c-odd} give
\[
 (2m+1)c_{2m+1}-(2m+2)c_{2m}
 =
 2\left[
 mA_m-
 \sum_{j=1}^m
 \frac{m+j+1}{j(2j+1)}A_{m-j}
 \right].
\]
Since $\alpha>1$, the sequence $(A_m)$ is increasing.  Moreover, for $1\leq j\leq m$,
\[
 \frac{m+j+1}{j(2j+1)}
 \leq
 \frac{m+1}{j(j+1)},
\]
because
\[
 \frac{m+1}{j(j+1)}
 -
 \frac{m+j+1}{j(2j+1)}
 =
 \frac{m-j}{(j+1)(2j+1)}
 \geq0.
\]
Therefore
\[
 \sum_{j=1}^m
 \frac{m+j+1}{j(2j+1)}
 \leq
 (m+1)\sum_{j=1}^m\frac1{j(j+1)}
 =m,
\]
and hence
\[
 c_{2m+1}
 \geq
 \frac{2m+2}{2m+1}c_{2m}.
\]
Combining this inequality with the even coefficient estimate gives
\[
 q_{2m+1}
 \leq
 D_\alpha(0)c_{2m+1}.
\]
Thus \eqref{eq:coefficient-target} holds for all $n\geq0$.

Finally,
\[
 q_1<2q_0=D_\alpha(0)c_1.
\]
Hence the coefficient of $r$ in
\[
 D_\alpha(0)L(r)(1-r^2)^{-\alpha}
 -
 \frac{D_\alpha(r)}{(1-r^2)^\alpha}
\]
is strictly positive, while all its coefficients are nonnegative.
Therefore \eqref{eq:D-over-L-zero} is strict for every $r>0$.
\end{proof}

\begin{lemma} \label{lem:log-interpolation}
If $1<\alpha\leq4/3$, then
\begin{equation}\label{eq:second-component-below}
 C_\alpha+K_\alpha^{(1)}\leq\frac32,
\end{equation}
with strict inequality when $1<\alpha<4/3$.
\end{lemma}

\begin{proof}
Let
\[
 \theta=3(\alpha-1)\in(0,1],
 \qquad
 \alpha=(1-\theta)\cdot1+\theta\cdot\frac43.
\]
By Fubini's theorem,
\begin{equation}\label{eq:D-double}
 D_\alpha(r)=
 \int_{0<s<t<1}\frac{t}{(1-rt)^2}
 \left(\frac{1-r^2}{1-s^2}\right)^\alpha\,\dd s\,\dd t.
\end{equation}
H\"older's inequality applied to \eqref{eq:D-double} gives
\[
 D_\alpha(r)\leq D_1(r)^{1-\theta}D_{4/3}(r)^\theta.
\]
After division by $L(r)$ and taking suprema, \cref{lem:alpha1-log-bound}
and \cref{lem:coefficient-domination} yield
\begin{equation}\label{eq:K-interpolation}
 K_\alpha^{(1)}
 \leq(\log2)^{1-\theta}\bigl(D_{4/3}(0)\bigr)^\theta
 \leq(1-\theta)\log2+\theta D_{4/3}(0),
\end{equation}
where  in the last step we used
\[
 x^{1-\theta}y^\theta
 \leq
 (1-\theta)x+\theta y,
 \qquad x,y>0,\quad 0\leq\theta\leq1.
\]

The map $\alpha\mapsto C_\alpha$ is strictly convex on $[1,4/3]$, because
\[
 C_\alpha''=
 \int_0^1(1-t)(1-t^2)^{-\alpha}
 \log^2(1-t^2)\,\dd t>0.
\]
Therefore
\begin{equation}\label{eq:C-convex}
 C_\alpha\leq(1-\theta)C_1+\theta C_{4/3}.
\end{equation}
Here $C_1=\log2$.  Moreover,  
\begin{align*}
 D_{4/3}(0)
 &=\frac{\sqrt\pi\,\Gamma(2/3)}{4\Gamma(7/6)}
 =-\frac{\sqrt\pi\,\Gamma(-1/3)}{2\Gamma(1/6)}
 =-d_{4/3},
\end{align*}
so \eqref{eq:C-gamma} gives
\begin{equation}\label{eq:critical-log-tie}
 C_{4/3}+D_{4/3}(0)=\frac32.
\end{equation}
Combining \eqref{eq:K-interpolation}--\eqref{eq:critical-log-tie},
\[
 C_\alpha+K_\alpha^{(1)}
 \leq(1-\theta)2\log2+\theta\frac32\leq\frac32.
\]
Since $2\log2<3/2$, the last inequality is strict when $\theta<1$.
\end{proof}

\begin{proof}[Proof of Theorem~\ref{thm:log-main}]
By \eqref{eq:K0-half}, the first quantity on the right-hand side of
\eqref{eq:log-component-formula} is $3/2$.

If $1<\alpha\leq4/3$, then \cref{lem:log-interpolation} gives
\eqref{eq:second-component-below}, with strict inequality when
$1<\alpha<4/3$.  This proves the first case of \eqref{eq:log-exact}.

For $4/3\leq\alpha<2$, \cref{lem:coefficient-domination} gives
\[
 K_\alpha^{(1)}=D_\alpha(0),
\]
whose exact value is given in \eqref{eq:D0-formula}.
Since both $C_\alpha$ and $D_\alpha(0)$ are strictly increasing in
$\alpha$, \eqref{eq:critical-log-tie} implies the second case of
\eqref{eq:log-exact} for $4/3<\alpha<2$.  

It remains to determine the norm-attaining functions.
Let \(f\in\mathcal B^\alpha\) have norm one and attain the operator norm.
Then $
 \abs{f(0)}+\norm{f}_{\alpha,*}=1$. 

Suppose first that \(1<\alpha<4/3\).  From
\eqref{eq:log-component-formula} and
\(1+K_\alpha^{(0)}=3/2\), we have
\[
 \begin{aligned}
 \norm{\mathcal Hf}_{\mathcal B^\alpha_{\log}}
 &\leq
 \frac32\abs{f(0)}
 +
 \bigl(C_\alpha+K_\alpha^{(1)}\bigr)
 \norm{f}_{\alpha,*}\\
 &=
 \frac32
 -
 \left(
 \frac32-C_\alpha-K_\alpha^{(1)}
 \right)
 \norm{f}_{\alpha,*}.
 \end{aligned}
\]
By \cref{lem:log-interpolation},
\[
 C_\alpha+K_\alpha^{(1)}<\frac32.
\]
Since \(f\) attains the operator norm \(3/2\), the preceding estimate
forces $
 \norm{f}_{\alpha,*}=0$. 
Hence \(f\) is constant, and its unit norm implies that
\[
 f(z)=\lambda,
 \qquad \lambda\in\T.
\]

Now suppose that \(4/3<\alpha<2\).  In this case,
\[
 C_\alpha+K_\alpha^{(1)}
 >
 1+K_\alpha^{(0)}
 =
 \frac32.
\]
Using  $
 \abs{f(0)}+\norm{f}_{\alpha,*}=1$, 
the norm estimate gives
\[
 \begin{aligned}
 \norm{\mathcal Hf}_{\mathcal B^\alpha_{\log}}
 &\leq
 \frac32\abs{f(0)}
 +
 \bigl(C_\alpha+K_\alpha^{(1)}\bigr)
 \norm{f}_{\alpha,*}\\
 &=
 C_\alpha+K_\alpha^{(1)}
 -
 \left(
 C_\alpha+K_\alpha^{(1)}-\frac32
 \right)\abs{f(0)}.
 \end{aligned}
\]
Since \(f\) attains the operator norm
\(C_\alpha+K_\alpha^{(1)}\), it follows that
\[
 f(0)=0,
 \qquad
 \norm{f}_{\alpha,*}=1.
\]
Moreover, equality in the norm estimate forces equality in the
estimate for the value at the origin.  Hence
\[
 \begin{aligned}
 C_\alpha
 &=
 \abs{\mathcal Hf(0)}
 =
 \abs{\int_0^1 f(t)\,\dd t} \leq
 \int_0^1\abs{f(t)}\,\dd t
 \leq
 \int_0^1P_\alpha(t)\,\dd t
 =
 C_\alpha.
 \end{aligned}
\]
Hence all the preceding inequalities are equalities.  Since
\[
 \abs{f(t)}\leq P_\alpha(t),
 \qquad 0<t<1,
\]
and the two sides have equal integrals, continuity gives
\[
 \abs{f(t)}=P_\alpha(t),
 \qquad 0<t<1.
\]
Equality in the integral triangle inequality then yields a
\(\lambda\in\T\) such that
\[
 f(t)=\lambda P_\alpha(t),
 \qquad 0<t<1.
\]
Since \(P_\alpha(t)=F_\alpha(t)\) on \((0,1)\), the identity theorem
gives $
 f=\lambda F_\alpha$.

Finally, let \(\alpha=4/3\).  By \eqref{eq:K0-half} and
\eqref{eq:critical-log-tie},
\[
 1+K_{4/3}^{(0)}
 =
 C_{4/3}+K_{4/3}^{(1)}
 =
 \frac32.
\]
Since \(f\) attains the operator norm and $
 \abs{f(0)}+\norm{f}_{4/3,*}=1$, 
equality must hold in both component estimates used to obtain
\eqref{eq:log-component-formula}.  In particular,
\[
 \abs{\mathcal Hf(0)}
 =
 \abs{f(0)}
 +
 C_{4/3}\norm{f}_{4/3,*}.
\]
Therefore
\[
 \begin{aligned}
 \abs{f(0)}
 +
 C_{4/3}\norm{f}_{4/3,*}
 &=
 \abs{\int_0^1f(t)\,\dd t} 
 \leq
 \int_0^1\abs{f(t)}\,\dd t 
 \leq 
 \abs{f(0)}
 +
 C_{4/3}\norm{f}_{4/3,*}.
 \end{aligned}
\]
Hence all the preceding inequalities are equalities. Continuity   implies that, for some
\(\lambda\in\T\),
\[
 f(t)
 =
 \lambda
 \left(
 \abs{f(0)}
 +
 \norm{f}_{4/3,*}P_{4/3}(t)
 \right),
 \qquad 0<t<1.
\]
Since \(P_{4/3}(t)=F_{4/3}(t)\) on \((0,1)\), the identity theorem gives
\[
 f(z)
 =
 \lambda
 \left(
 \abs{f(0)}
 +
 \norm{f}_{4/3,*}F_{4/3}(z)
 \right),
 \qquad z\in\D.
\]
Writing $ a=\abs{f(0)}$, 
we have
\[
 \norm{f}_{4/3,*}=1-a,
 \qquad 0\leq a\leq1,
\]
and therefore
\[
 f(z)
 =
 \lambda\bigl(a+(1-a)F_{4/3}(z)\bigr).
\]

Conversely, a direct substitution shows that every function listed in
\eqref{eq:log-extremals} has norm one and attains the operator norm.
This proves \eqref{eq:log-extremals}.
\end{proof}

\section{Norm of the Hilbert matrix
\texorpdfstring{$\norm{\mathcal H}_{H^1 \to  H^1_{\log}}$}
{H1-log}} 
\label{sec:H1-H1log}
  Ye and Zheng
\cite[Corollary~4.1]{YeZheng2025} proved
\[
 \pi\leq
 \norm{\mathcal H}_{H^1\to H^1_{\log}}
 \leq2\pi\log 2
\]
and conjectured \cite[Conjecture~1]{YeZheng2025} that the norm equals
$\pi$. Their lower bound will be used below.  For the upper bound, we use the Fej\'er--Riesz inequality to control
the radial integral of \(f\) by its \(H^1\)-norm, and show that the
angular average of the Cauchy kernel is bounded by the logarithmic
weight \(L\).

We first establish an  estimate for the integral kernel.

\begin{lemma}\label{lem:H1-kernel}
For \(0\leq s<1\),
\[
 \frac{1}{2\pi}
 \int_0^{2\pi}
 \frac{\dd\theta}{|1-se^{i\theta}|}
 \leq
 L(s)
 =
 \log\frac{e}{(1-s)^2}.
\]
\end{lemma}

\begin{proof}
For \(0<s<1\), the estimate recorded in
\cite[Lemma~4.2(3)]{YeZheng2025} gives
\[
 \frac{1}{2\pi}
 \int_0^{2\pi}
 \frac{\dd\theta}{|1-se^{i\theta}|}
 \leq
 \frac{1}{s^2}\log\frac{1}{1-s^2}.
\]
It therefore suffices to show that
\[
 \frac{1}{s^2}\log\frac{1}{1-s^2}
 \leq L(s).
\]
Indeed,
\[
 \frac{1}{s^2}\log\frac{1}{1-s^2}
 =
 1+\sum_{m=1}^{\infty}\frac{s^{2m}}{m+1},
\]
whereas
\[
 L(s)
 =
 1-2\log(1-s)
 =
 1+\sum_{m=1}^{\infty}\frac{s^{2m}}{m}
 +2\sum_{m=0}^{\infty}\frac{s^{2m+1}}{2m+1}.
\]
Since
\[
 \frac{1}{m+1}\leq\frac{1}{m},
 \qquad m\geq1,
\]
comparison of the nonnegative coefficients gives the desired inequality.
For \(s=0\), both sides are equal to \(1\).
\end{proof}

\begin{proof}[Proof of Theorem~\ref{thm:exact-H1-H1log}]
By the lower-bound estimate of Ye and Zheng
\cite[Theorem~4.2]{YeZheng2025},
\begin{equation}\label{eq:known-lower-bound}
 \norm{\mathcal H}_{H^1\to H^1_{\log}}
 \geq\pi.
\end{equation}
It therefore remains to prove the reverse inequality.

Let $f\in H^1$. By the classical Fej\'er--Riesz inequality
\cite[Theorem~3.13, p.~46]{Duren1970},  
\[
 \int_{-1}^{1}|f(t)|\,\dd t
 \leq
 \frac12
 \int_0^{2\pi}|f(e^{i\theta})|\,\dd\theta
 =
 \pi\norm{f}_{H^1}.
\]
Consequently,
\begin{equation}\label{eq:FR-one-sided}
 \int_0^1|f(t)|\,\dd t
 \le
 \pi\norm{f}_{H^1}.
\end{equation}
In particular, $f\in L^1(0,1)$.

Fix $0\le r<1$. Since $ |1-rte^{i\theta}|
 \ge 1-rt
 \ge 1-r$, 
we have
\[
 \frac{|f(t)|}{|1-rte^{i\theta}|}
 \le
 \frac{|f(t)|}{1-r}.
\]
The right-hand side is integrable on $[0,1]$ by
\eqref{eq:FR-one-sided}. By Fubini's theorem and \eqref{eq:H-definition}, we obtain
\[
 \begin{aligned}
 M_1(r,\mathcal Hf)
 &=
 \frac{1}{2\pi}
 \int_0^{2\pi}
 \left|
 \int_0^1\frac{f(t)}{1-rte^{i\theta}}\,\dd t
 \right|\dd\theta \\
 &\le
 \frac{1}{2\pi}
 \int_0^{2\pi}\int_0^1
 \frac{|f(t)|}{|1-rte^{i\theta}|}\,\dd t\,\dd\theta \\
 &=
 \int_0^1|f(t)|
 \left(
 \frac{1}{2\pi}
 \int_0^{2\pi}
 \frac{\dd\theta}{|1-rte^{i\theta}|}
 \right)\dd t.
\end{aligned}
\]
By \cref{lem:H1-kernel} and the monotonicity of \(L\),
\[
 \frac{1}{2\pi}
 \int_0^{2\pi}
 \frac{\dd\theta}{|1-rte^{i\theta}|}
 \leq L(rt)\leq L(r),
 \qquad 0\leq t\leq1.
\]
Hence
\[
 \begin{aligned}
 M_1(r,\mathcal Hf)
 &\leq
 L(r)\int_0^1|f(t)|\,\dd t \leq
 \pi L(r)\norm{f}_{H^1}.
 \end{aligned}
\]
where the last inequality follows from \eqref{eq:FR-one-sided}. Thus
\[
 L(r)^{-1}M_1(r,\mathcal Hf)
 \le
 \pi\norm{f}_{H^1}.
\]
Taking the supremum over $0\le r<1$, we conclude that
\[
 \norm{\mathcal Hf}_{H^1_{\log}}
 \le
 \pi\norm{f}_{H^1}.
\]
Hence
\begin{equation}\label{eq:new-upper-bound}
 \norm{\mathcal H}_{H^1\to H^1_{\log}}
 \le \pi.
\end{equation}
Combining \eqref{eq:known-lower-bound} and \eqref{eq:new-upper-bound}
gives
\[
 \norm{\mathcal H}_{H^1\to H^1_{\log}}
 =\pi.
\]
\end{proof}

\section{Norm of the Hilbert matrix  
\texorpdfstring{$\norm{\mathcal H}_{H^\infty_\alpha\to\mathcal B^{\alpha+1}}$}
{H-alpha-infinity to B-alpha+1}}
\label{sec:korenblum-bloch}

Hu and Ye \cite[Theorem~7.1]{HuYe2025} proved the exact norm for
$0<\alpha\leq2/3$ and obtained two-sided estimates for
$2/3<\alpha<1$.  Their lower-bound argument remains valid throughout
$0<\alpha<1$.  Thus the remaining problem is to prove the corresponding
sharp upper bound when $2/3<\alpha<1$.

Set
\begin{equation}\label{eq:KB-A-V}
 A_\alpha
 =\int_0^1\frac{\dd t}{(1-t^2)^\alpha},
 \qquad
 V_\alpha(r)
 =(1-r^2)^{\alpha+1}
 \int_0^1\frac{t\,\dd t}
 {(1-rt)^2(1-t^2)^\alpha},
 \quad 0\leq r<1.
\end{equation}

The estimates in the proof of
\cite[Theorem~7.1]{HuYe2025}, together with the same test function,
give
\[
 \norm{\mathcal H}_{H^\infty_\alpha\to\mathcal B^{\alpha+1}}
 =
 A_\alpha+\sup_{0\leq r<1}V_\alpha(r).
\] 
It therefore remains to determine the supremum of \(V_\alpha\).
Using the substitutions
\[
 r=\tanh R,\qquad t=\tanh x,
\]
we show that \(V_\alpha\) is strictly increasing on \([0,1)\) for every
\(0<\alpha<1\). Consequently,
\[
 \sup_{0\leq r<1}V_\alpha(r)
 =
 \lim_{r\to1^-}V_\alpha(r).
\] 
\begin{proposition}\label{prop:KB-exact-reduction}
Let \(0<\alpha<1\). Then
\begin{equation}\label{eq:KB-exact-reduction}
 \norm{\mathcal H}_{H^\infty_\alpha\to\mathcal B^{\alpha+1}}
 =
 A_\alpha+\sup_{0\leq r<1}V_\alpha(r).
\end{equation}
Moreover, the function
\begin{equation}\label{eq:KB-extremal}
 k_\alpha(z)=(1-z^2)^{-\alpha}
\end{equation}
has unit norm in \(H^\infty_\alpha\) and attains the operator norm.
\end{proposition}
\begin{proof}
By the estimates and the test function in the proof of
\cite[Theorem~7.1]{HuYe2025}, and after rewriting their expressions
in the notation \eqref{eq:KB-A-V}, we obtain
\[
 \norm{\mathcal H}_{H^\infty_\alpha\to\mathcal B^{\alpha+1}}
 =
 A_\alpha+\sup_{0\leq r<1}V_\alpha(r).
\]
The same test function $
 k_\alpha(z)=(1-z^2)^{-\alpha}$ 
has unit norm in \(H^\infty_\alpha\) and attains the operator norm.
\end{proof}

\begin{lemma}\label{lem:KB-V-monotone}
Let $0<\alpha<1$.  Then $V_\alpha$ is strictly increasing on $[0,1)$ and
\begin{equation}\label{eq:KB-V-limit}
 \lim_{r\to1^-}V_\alpha(r)
 =2B(1+\alpha,1-\alpha)
 =\frac{2\pi\alpha}{\sin(\pi\alpha)}.
\end{equation}
Consequently,
\begin{equation}\label{eq:KB-V-sup}
 \sup_{0\leq r<1}V_\alpha(r)
 =\frac{2\pi\alpha}{\sin(\pi\alpha)}.
\end{equation}
\end{lemma}

\begin{proof}
Write
\[
 r=\tanh R,\qquad t=\tanh x,
 \qquad R,x\geq0.
\]
Using
\[
 1-\tanh R\tanh x
 =\frac{\cosh(R-x)}{\cosh R\cosh x},
 \qquad
 \dd t=\operatorname{sech}^2x\,\dd x,
\]
we obtain
\begin{equation}\label{eq:KB-V-hyperbolic}
 V_\alpha(\tanh R)
 =\int_0^\infty
 \tanh x
 \left(\frac{\cosh x}{\cosh R}\right)^{2\alpha}
 \operatorname{sech}^2(R-x)\,\dd x.
\end{equation}
Putting $y=x-R$ gives
\begin{equation}\label{eq:KB-V-shifted}
 V_\alpha(\tanh R)
 =\int_{-R}^\infty Q_R(y)\,\dd y,
\end{equation}
where
\[
 Q_R(y)
 =\tanh(R+y)
 \left(\frac{\cosh(R+y)}{\cosh R}\right)^{2\alpha}
 \operatorname{sech}^2y,
 \qquad y>-R.
\]
Fix $y>-R$ and put $u=R+y>0$.  Since $Q_R(y)>0$,
\begin{align}
 \frac{\partial}{\partial R}\log Q_R(y)
 &=\frac{\operatorname{sech}^2u}{\tanh u}
   +2\alpha(\tanh u-\tanh R)\notag\\
 &=\frac1{\sinh u\cosh u}
   +2\alpha(\tanh u-\tanh R).
 \label{eq:KB-logder}
\end{align}
If $y\geq0$, then $u\geq R$, and the right-hand side of
\eqref{eq:KB-logder} is strictly positive.  If $-R<y<0$, then
$0<u<R$.  Since $\tanh R<1$,
\[
 \tanh R-\tanh u<1-\tanh u,
\]
and therefore
\begin{align*}
 \frac{\partial}{\partial R}\log Q_R(y)
 &>\frac1{\sinh u\cosh u}-2\alpha(1-\tanh u)\\
 &=(1-\tanh u)
 \left(\frac{e^u}{\sinh u}-2\alpha\right).
\end{align*}
Now
\[
 \frac{e^u}{\sinh u}=\frac{2}{1-e^{-2u}}>2,
\]
so, because $0<\alpha<1$,
\[
 \frac{\partial}{\partial R}\log Q_R(y)
 >2(1-\alpha)(1-\tanh u)>0.
\]
We now compare two values of the parameter.  Let $R_2>R_1\ge0$.
Since $Q_R(y)$ is strictly increasing in $R$ whenever $y>-R$, we have
\begin{align*}
&V_\alpha(\tanh R_2)-V_\alpha(\tanh R_1)\\
&\quad=
\int_{-R_2}^{-R_1}Q_{R_2}(y)\,\dd y
+\int_{-R_1}^{\infty}
\bigl(Q_{R_2}(y)-Q_{R_1}(y)\bigr)\,\dd y
>0.
\end{align*}
Thus $R\mapsto V_\alpha(\tanh R)$ is strictly increasing.  Since
$R\mapsto\tanh R$ is strictly increasing, $V_\alpha$ is strictly
increasing on $[0,1)$.

For each fixed \(y\in\mathbb R\), once \(R>-y\), the monotonicity
proved above shows that \(R\mapsto Q_R(y)\) is increasing.  Moreover,
\[
 \lim_{R\to\infty}Q_R(y)
 =
 e^{2\alpha y}\operatorname{sech}^2y.
\]
Consequently, for every \(R>-y\),
\[
 0<Q_R(y)
 \leq
 e^{2\alpha y}\operatorname{sech}^2y.
\]
Since
\[
 e^{2\alpha y}\operatorname{sech}^2y\in L^1(\mathbb R),
 \qquad 0<\alpha<1,
\]
we claim that
\[
 \lim_{R\to\infty}V_\alpha(\tanh R)
 =
 \int_{-\infty}^{\infty}
 e^{2\alpha y}\operatorname{sech}^2y\,\dd y.
\]
Indeed, fix $M>0$.  If $R>M$, then $[-M,M]\subset(-R,\infty)$.
Hence, by dominated convergence,
\[
 \lim_{R\to\infty}
 \int_{-M}^{M}Q_R(y)\,\dd y
 =
 \int_{-M}^{M}
 e^{2\alpha y}\operatorname{sech}^2y\,\dd y.
\]
On the other hand, the part outside $[-M,M]$ is uniformly controlled in
$R$ by
\[
 \int_{\{y>-R,\ |y|>M\}}Q_R(y)\,\dd y
 \le
 \int_{|y|>M}
 e^{2\alpha y}\operatorname{sech}^2y\,\dd y.
\]
Since the function on the right is integrable on $\mathbb R$, its tail
tends to zero as $M\to\infty$.  Thus, first choosing $M$ sufficiently
large and then letting $R\to\infty$ gives
\[
 \lim_{R\to\infty}
 \int_{-R}^{\infty}Q_R(y)\,\dd y
 =
 \int_{-\infty}^{\infty}
 e^{2\alpha y}\operatorname{sech}^2y\,\dd y.
\]
By \eqref{eq:KB-V-shifted}, this is precisely
\[
 \lim_{R\to\infty}V_\alpha(\tanh R)
 =
 \int_{-\infty}^{\infty}
 e^{2\alpha y}\operatorname{sech}^2y\,\dd y.
\]

With
\[
 s=\frac{1+\tanh y}{2},
 \qquad
 \dd s=\frac12\operatorname{sech}^2y\,\dd y,
 \qquad
 e^{2y}=\frac{s}{1-s},
\]
we obtain
\begin{align*}
 \int_{-\infty}^{\infty}
 e^{2\alpha y}\operatorname{sech}^2y\,\dd y
 &=2\int_0^1s^\alpha(1-s)^{-\alpha}\,\dd s\\
 &=2B(1+\alpha,1-\alpha).
\end{align*}
Finally, Euler's reflection formula gives
\[
 B(1+\alpha,1-\alpha)
 =\Gamma(1+\alpha)\Gamma(1-\alpha)
 =\frac{\pi\alpha}{\sin(\pi\alpha)},
\]
which proves \eqref{eq:KB-V-limit} and \eqref{eq:KB-V-sup}.
\end{proof}

\begin{proof}[Proof of Theorem~\ref{thm:KB-complete}]
By \cref{prop:KB-exact-reduction,lem:KB-V-monotone},
\[
 \norm{\mathcal H}_{H^\infty_\alpha\to\mathcal B^{\alpha+1}}
 =A_\alpha+\frac{2\pi\alpha}{\sin(\pi\alpha)}.
\]
The substitution $u=t^2$ gives
\[
 A_\alpha
 =\frac12B\left(\frac12,1-\alpha\right)
 =\frac{\sqrt\pi\,\Gamma(1-\alpha)}
        {2\Gamma(3/2-\alpha)}.
\]
The assertion concerning the extremal follows from
\cref{prop:KB-exact-reduction}.
\end{proof}

\begin{corollary}\label{cor:KB-boundedness}
For $\alpha>0$, the Hilbert matrix operator
\[
 \mathcal H:H^\infty_\alpha\longrightarrow\mathcal B^{\alpha+1}
\]
is bounded if and only if $0<\alpha<1$.  On this range its norm is given
by \eqref{eq:KB-complete-norm}.
\end{corollary}

\begin{proof}
The boundedness and norm formula for $0<\alpha<1$ follow from
\cref{thm:KB-complete}.  The failure of boundedness for $\alpha\geq1$
was proved by Hu and Ye in \cite[Theorem~7.1]{HuYe2025}.
\end{proof}

\subsection*{Author contributions}
Puyu Cui, Zhaopeng Lin, Yufeng Lu, and Shanli Ye contributed to the conception,
development, and writing of the manuscript. All authors read and approved
the final manuscript.

\subsection*{Funding}
Y. Lu was supported by the National Natural Science Foundation of China
(Grant Nos. 12031002 and 12671149).
S. Ye was supported by the Zhejiang Provincial Natural Science Foundation
(Grant No. LY23A010003) and the National Natural Science Foundation of China
(Grant No. 11671357).

\subsection*{Conflict of interest}
The authors have no conflict of interest to declare that are relevant to the content of this article.
 
\subsection*{Data Availability}
Data sharing is not applicable to this article,
as no datasets were generated or analysed
during the current study.

\subsection*{Declaration of generative AI and AI-assisted technologies
in the manuscript preparation process}

In Section~4, AI assistance was used in the derivation
of the coefficient inequalities in \eqref{eq:coefficient-target},
including the derivation of the inequalities
for even and odd coefficient indices. 
All AI-assisted mathematical arguments were subsequently reviewed
and independently verified by the authors, who take full
responsibility for the accuracy and integrity of the manuscript.
\bibliographystyle{amsplain}
\bibliography{references}

\end{document}